\documentclass[12pt,leqno]{article} 
\usepackage{amsmath,amssymb,amsthm,amsfonts, indentfirst}
\usepackage[dvipdfmx]{graphicx}
\usepackage{enumerate,color,bm}
\makeatletter
\def\@cite#1#2{[{{\bfseries #1}\if@tempswa , #2\fi}]}
\renewcommand{\section}{%
\@startsection{section}{1}{\z@}
{0.5truecm plus -1ex minus -.2ex}%
{1.0ex plus .2ex}{\bfseries\large}}
\def\@seccntformat#1{\csname the#1\endcsname.\ }
\makeatother

\numberwithin{equation}{section} 
\newtheorem{thm}{Theorem}[section]
\newtheorem{corollary}[thm]{Corollary}
\newtheorem{lem}[thm]{Lemma}

\theoremstyle{definition}

\newtheorem{remark}{Remark}[section]

\newcommand{\ep}{\varepsilon}
\newcommand{\pa}{\partial}

\newcommand{\tmax}{T_{\rm max}}
\newcommand{\lp}[2]{\|#2\|_{L^{#1}(\Omega)}}

\newcommand{\io}{\int_\Omega }

\begin{document}
\footnote[0]
    {2010{\it Mathematics Subject Classification}\/. 
    Primary: 35B40; Secondary: 35K55, 35Q30, 92C17.
    }
\begin{center}
    \Large{{\bf 
    Extensibility criterion ruling out gradient blow-up\\ 
    in a quasilinear degenerate chemotaxis system\\
    with flux limitation
    }}
\end{center}
\vspace{5pt}
\begin{center}
     Masaaki Mizukami%
     \footnote{Corresponding author}%
     \footnote{Partially supported by JSPS Research Fellowship 
    for Young Scientists, No.\ 17J00101 }, 
    Tatsuhiko Ono,
    Tomomi Yokota%
   \footnote{Partially supported by Grant-in-Aid for
    Scientific Research (C), No.\ 16K05182.}%
   \footnote[0]{
    E-mail: 
     {\tt masaaki.mizukami.math@gmail.com}, 
    \ {\tt wpsonotatsu45@gmail.com}, \\
    \ {\tt yokota@rs.kagu.tus.ac.jp}
    }\\
    \vspace{12pt}
    Department of Mathematics, 
    Tokyo University of Science\\
    1-3, Kagurazaka, Shinjuku-ku, 
    Tokyo 162-8601, Japan\\
    \vspace{2pt}
\end{center}
\begin{center}    
    \small \today
\end{center}

\vspace{2pt}
\newenvironment{summary}
{\vspace{.5\baselineskip}\begin{list}{}{%
     \setlength{\baselineskip}{0.85\baselineskip}
     \setlength{\topsep}{0pt}
     \setlength{\leftmargin}{12mm}
     \setlength{\rightmargin}{12mm}
     \setlength{\listparindent}{0mm}
     \setlength{\itemindent}{\listparindent}
     \setlength{\parsep}{0pt}
     \item\relax}}{\end{list}\vspace{.5\baselineskip}}
\begin{summary}
{\footnotesize {\bf Abstract.}
    This paper deals with the 
    quasilinear degenerate chemotaxis system with flux limitation
        \begin{equation*}   
     \begin{cases}
         u_t = \nabla\cdot\left(\dfrac{u^p
         \nabla u}{\sqrt{u^2 + |\nabla u|^2}}
         \right) -\chi \nabla\cdot\left(
         \dfrac{u^q\nabla v}{\sqrt{1 + 
         |\nabla v|^2}}\right),
\\[2mm]
         0 = \Delta v - \mu + u        
     \end{cases}
 \end{equation*}
    under no-flux boundary conditions in balls 
    $\Omega\subset\mathbb{R}^n$, and the initial condition 
          $u|_{t=0}=u_0$
     for a radially symmetric and positive initial data 
     $u_0\in C^3(\overline{\Omega})$, 
     where $\chi>0$ and $\mu:=\frac{1}{|\Omega|}\int_{\Omega}u_0$. 
     Bellomo--Winkler 
     (Comm.\  Partial Differential Equations;2017;42;436--473) 
     proved local existence of unique 
     classical solutions and extensibility criterion ruling out gradient 
     blow-up 
    as well as global existence and boundedness of solutions when $p=q=1$ 
   under some conditions for $\chi$ and $\int_\Omega u_0$. 
This paper derives local existence and extensibility criterion 
ruling out gradient blow-up when $p,q\geq 1$, 
and moreover shows
 global existence and boundedness of solutions when 
 $p>q+1-\frac{1}{n}$. 
}%
\end{summary}
\vspace{10pt}

\newpage

\section{Introduction and results} \label{Sec1}
 In this paper we consider the following 
 quasilinear degenerate chemotaxis system with flux limitation: 
 \begin{equation}\label{P}     
     \begin{cases}
         u_t = \nabla\cdot\left(\dfrac{u^p
         \nabla u}{\sqrt{u^2 + |\nabla u|^2}}
         \right) -\chi \nabla\cdot\left(
         \dfrac{u^q\nabla v}{\sqrt{1 + 
         |\nabla v|^2}}\right),
         &x\in \Omega,\ t>0,
\\[2mm]
         0 = \Delta v - \mu + u, 
         &x\in \Omega,\ t>0,         
\\[2mm]
          \left(\dfrac{u^p\nabla u}{\sqrt{u^2 +
          |\nabla u|^2}} - \chi
          \dfrac{u^q\nabla v}{\sqrt{1 +
          |\nabla v|^2}}\right)\cdot\nu = 0, 
         &x \in \partial\Omega,\ t>0,
\\[5mm]
         u(x,0) = u_0(x),
         &x \in \Omega,
     \end{cases}
 \end{equation}
\noindent
where $\Omega= B_R(0)\subset\mathbb{R}^n$, $n\geq1$ ,
 $\nu$ is the outward normal vector to $\pa\Omega$ and $\chi>0$ 
 indicates the strength of chemotactic cross-diffusion. 
  The initial data $u_0$ is assumed to be a 
   function satisfying  
\begin{align}\label{u0}
u_0\in C^3(\overline{\Omega})
\quad\mbox{is radially symmetric and positive in}
\ \overline{\Omega}\ \mbox{with}
\ \frac{\partial u_0}{\partial \nu}=0\ \mbox{on}\ \partial\Omega,
\end{align} 
so that the spatial average 
\begin{align}
\mu:=\frac{1}{|\Omega|}\int_{\Omega}u_0(x)\,dx
\end{align} 
is positive. 
%

From a point of the biological view, 
this problem \eqref{P} describes 
the evolution of a species which has chemotaxis, 
where chemotaxis is the property such that species 
move towards higher concentration of a chemical substance. 
The unknown function 
$u(x,t)$ denotes the population density of species 
and the unknown function 
$v(x,t)$ represents the concentration of the chemical substance 
at place $x\in\Omega$ and time $t\geq 0$. 
In the problem \eqref{P} the terms 
 $ \nabla\cdot\big(\frac{u^p \nabla u}
{\sqrt{u^2 + |\nabla u|^2}} \big)$
 and 
 $-\chi \nabla\cdot\big( \frac{u^q\nabla v}
{\sqrt{1 + |\nabla v|^2}} \big)$
describe the effect of diffusion and 
the effect of chemotactic interaction, respectively. 
Moreover, the flux limitation provides the situation such that 
species can move through some specific way, e.g., the border of the cells, with finite speed of propagation.   
(for more detail, see \cite{B-B-T-W}).
Here the problem \eqref{P} is a special case of 
the following generalized problem of the chemotaxis system 
such that the first and second equations of \eqref{P} are replaced with 
\begin{equation}\label{Porigin} 
 \begin{cases} 
 u_t =
  \nabla\cdot\left(D_u(u,v)\dfrac{u\nabla u}{\sqrt{u^2 + \frac{{\nu}^2}{c^2}|\nabla u|^2}} 
   -S(u,v)\dfrac{u\nabla v}{\sqrt{1+|\nabla v|^2}}\right)
   +H_1(u,v), &
 \\ 
  v_t =  D_v \Delta v + H_2 (u,v)
  &
  \end{cases}
\end{equation}
where $D_u$ and $D_v$ denote the property of 
cell's and chemoattractant's diffusion, respectively, 
and $S$ shows the chemotactic sensitivity 
as well as $H_1, H_2$ represent interactions. 
Here, $\nu$ and $c$ are quantities which 
describe kinematic viscosity and maximum speed of propagation.  

From a point of the mathematical view, 
because of the difficulties of the flux limitation, 
good functions such as 
a Lyapunov function and an energy function seem not to be found.  
Bellomo--Winkler \cite{B-W} made a breakthrough in this area 
by considering the problem which is \eqref{P} 
with $p=q=1:$      
 \begin{align}\label{P11}     
 \begin{cases}
  u_t = \nabla\cdot\left(\dfrac{u\nabla u}{\sqrt{u^2 + |\nabla u|^2}} \right)
   -\chi \nabla\cdot\left(\dfrac{u\nabla v}{\sqrt{1 +|\nabla v|^2}}\right),
\\[2mm]
     0 = \Delta v - \mu + u;         
 \end{cases}
 \end{align}
in \cite{B-W} 
local existence with extensibility 
criterion and global existence of bounded radial solutions 
were shown 
under some conditions for 
$\chi$ and $\int_\Omega u_0$. 
Moreover, Bellomo--Winkler \cite{B-W;blowup} established 
existence of an initial data such that the corresponding 
solution blows up in finite time 
under some conditions for $\chi$ and $\int_\Omega u_0$. 
Even though Bellomo--Winkler \cite{B-W,B-W;blowup} 
overcame the difficulties come from the flux limitation in the special setting, because of difficulties of the problem \eqref{Porigin}, 
there still are only two previous results about the chemotaxis system with flux limitation. 
%

On the other hand, the problem \eqref{Porigin} without flux limitation and with some special setting of $D_u,D_v, S,H_1,H_2$, 
\begin{align}\label{PKS}     
 \begin{cases}
 u_t = \nabla\cdot\left(u^{p-1} \nabla u - u^{q-1}\nabla v\right), 
\\
     v_t = \Delta v -v + u       
 \end{cases}
 \end{align}
is called a chemotaxis system  
and is investigated intensively. 
The system \eqref{PKS} with $p=1$ and $q=2$ is first proposed by 
Keller--Segel \cite{K-S}, and there are several results on this problem; 
global existence and boundedness can be found in 
\cite{Cao,N-S-Y,O-Y}; 
existence of blow-up solutions is in \cite{H-W,M-W,W-2013}. 
On the other hand, 
Hillen--Painter \cite{H-P} proposed the degenerate chemotaxis system, 
that is, the problem \eqref{PKS} 
with $p> 1$ and $q>2$, 
to describe a sensitive dynamics in phenomena.  
In the degenerate chemotaxis system, it is known that 
the relation between $p$ and $q$ determines the properties of solutions to the system; 
Sugiyama--Kunii \cite{S-K} first dealt with 
the degenerate chemotaxis system in the case that $\Omega=\mathbb{R}^n$ 
and obtained global existence of solutions when $q\leq m$; 
the condition for global existence was extended 
from $q\leq m$ to $q<m+\frac{2}{N}$ 
in \cite{I-Y2012} and 
their boundedness was obtained in \cite{I-Y2012-2};
global existence and boundedness in the case that $\Omega$ is a bounded domain can be found in \cite{I-S-Y};
in the case that $q > m + \frac{2}{N}$ existence of blow-up solutions 
was established in \cite{H-I-Y}. 
%

%
In view of the study of the chemotaxis system, 
the system \eqref{P} is a natural and meaningful problem 
as a generalization of the problem \eqref{P11}; 
thus to consider the system \eqref{P} is an important 
step to consider the system \eqref{Porigin}. 
Therefore the main purpose of this paper is to obtain
 the following two results about the problem \eqref{P}:
 \begin{itemize}
 \item local existence and extensibility criterion ruling out 
 gradient blow-up,
 \item global existence and boundedness of solutions 
 under some condition for $p$ and $q$.
 \end{itemize}
\medskip
Here the quantities $u^{p-1}$ and $u^{q-1}$ with $p\neq 1$ or $q\neq 1$ in the diffusion term and the chemotaxis term, respectively, 
destroy the mathematical structure of the system with $p=q=1$. Indeed, because of these quantities, 
we could not employ the same argument as in \cite{B-W} which is based on the comparison principle; 
in particular, since there are new {\it nonlinear} terms 
in some parabolic operator, 
a comparison function used in \cite{B-W}, which is a solution to some {\it linear} ordinary differential equation could not work well. 
Thus in order to attain the purposes of this work, 
we need to deal with the difficulties of the new quantities which come from the nonlinear terms.  

Now we state the main theorems.
 The first result is 
concerned with local existence and extensibility criterion. 
%
\medskip
\begin{thm}\label{mainthm1}
Suppose that $p,q \geq 1$ and that $u_0$ complies with \eqref{u0}. 
 Then there exist $\tmax\in(0,\infty]$ and a pair $(u,v)$ of 
 positive radially symmetric functions 
 \begin{align*}
u\in C^{2,1}(\overline{\Omega}\times[0,\tmax))
\quad \mbox{and}\quad
v\in C^{2,0}(\overline{\Omega}\times[0,\tmax))
\end{align*}
which solve 
\eqref{P} classically in $\Omega\times(0,\tmax)$, and moreover 
$u$ satisfies the following extensibility criterion\/{\rm :}  
\begin{align}\label{extensibility;criterion}
\mbox{if}\quad\tmax<\infty,\quad\mbox{then}\quad
\limsup_{t\nearrow\tmax}\|u(\cdot,t)\|_{L^{\infty}(\Omega)}
=\infty.
\end{align}
 \end{thm}
\medskip
%
Next, aided by extensibility criterion from Theorem \ref{mainthm1}, 
we obtain global existence and boundedness of solutions.  
\begin{thm}\label{mainthm2}
Assume that $u_0$ satisfies \eqref{u0}, and let $p,q \geq 1$ be constants such that 
\begin{align}\label{pq-Assumption}
p>q+1-\frac{1}{n}.
\end{align}
Then the problem \eqref{P} possesses a global classical 
solution $(u,v)$ which is a pair of radially symmetric functions 
satisfying that 
\begin{align*}
u\in C^{2,1}(\overline{\Omega}\times[0,\infty))
\quad\mbox{and}\quad
v\in C^{2,0}(\overline{\Omega}\times[0,\infty))
\end{align*}
 and that there exists $C>0$ such that
 \begin{align*}
\|u(\cdot,t)\|_{L^{\infty}(\Omega)}\leq C 
\quad\mbox{and}\quad
\|v(\cdot,t)\|_{L^{\infty}(\Omega)}\leq C
\end{align*}
for all $t>0$. 
\end{thm}
\begin{remark}
This theorem shows global existence of solutions to \eqref{P} 
when $p> q+1-\frac{1}{n}$.
On the other hand, in \cite{CMY} 
existence of blow-up solutions is obtained when $p\le q$.
Here there is a gap between these results; 
in the case that $q < p \le q+1-\frac 1n$, 
 behaviour
 of solutions is an open problem except the case that 
$n=1$. 

\end{remark}
\medskip
In Theorem \ref{mainthm1}, extensibility criterion 
 \eqref{extensibility;criterion} foresees to establish not only 
 the results for
 global existence and boundedness of solutions 
  (Theorem \ref{mainthm2}) but also the result for
  finite time blow-up of solutions
  (see \cite{CMY}), while extensibility criterion  
  in the result on local existence via the standard manner
   (see Lemma 2.1)
    is written as 
   \begin{align*}
\mbox{if}\ \ \tmax<\infty,\ \mbox{then either}
\ 
\liminf_{t\nearrow\tmax}\inf_{x\in\Omega}u(x,t)=0
\ \mbox{or}\ 
\limsup_{t\nearrow\tmax}\|u(\cdot,t)\|_{W^{1,\infty}(\Omega)}
=\infty.
\end{align*}   
   This includes possibility
   of extinction and gradient blow-up of solutions. 
 Therefore, the essential part is to obtain 
 extensibility criterion ruling out this possibility. 
 Especially, the main difficulty in the proof is to show the estimate 
 $\|\nabla u(\cdot, t)\|_{L^\infty(\Omega)}\leq C$ with some $C>0$. 
 We show this key estimate via using comparison arguments
with a new comparison function.  

First, in Section \ref{Sec2}, we calculate a partial derivative of $u_t$  
 with respect to $r$ and introduce an operator $\mathcal{P}$.
 Since $u_{rt}$ has new terms such as 
 \[p(p-1)\frac{u^{p-2}u^7_r} {{\sqrt{u^2+u^2_r}}^5}
                  -q(q-1)\chi\frac{u^{q-2}u^2_rv_r}{\sqrt{1+v^2_r}},\]
 it is necessary to introduce a new operator 
 which is different from \cite{B-W} such that 
\begin{align*}
   (\mathcal{P}\varphi)(r,t)
   :={\varphi}_t-A_1(r,t){\varphi}_{rr}-A_2(r,t){\varphi}_r
   -a_3(r,t){\varphi}^2-A_3(r,t)\varphi-A_4(r,t),
   \end{align*}
   with a new term $a_3(r,t){\varphi}^2$.
Accordingly, we are forced to change a comparison function.
Section \ref{Sec3} is devoted to obtaining a lower estimate for $u$ which implies that extinction of solutions has never happened. 
In Section \ref{Sec4}, to obtain a lower estimate for $u_r$, 
we define a new comparison function $\underline{\varphi}$ 
by connecting parts of a tangent function and their transitions 
 which satisfy
 some ordinary differential equation. 
Here, since tangent functions have asymptotic lines, the arguments become more 
sensitive than \cite{B-W}. 
In Section \ref{Sec5} we establish
 an upper estimate for $u_r$ and show Theorem \ref{mainthm1}. 

In Theorem \ref{mainthm2}, 
the strategy for the proof of boundedness of $u$ 
is to establish an $L^\infty$-estimate for $u$. 
In Section \ref{Sec6}, 
using 
\[
\nabla\cdot \left(\frac{\nabla v}{\sqrt{1+|\nabla v|^2}}\right)
=\Delta v \frac{1}{\sqrt{1+|\nabla v|^2}}+
\nabla v \cdot \nabla\left( \frac{1}{\sqrt{1+|\nabla v|^2}}\right)
\]
and the fact that $u$ is radially symmetric 
and aided by our condition $p>q+1-\frac{1}{n}$, 
from utilizing the energy function $\int_\Omega u^m$ for $m\ge 1$ 
we obtain boundedness of solutions and 
establish Theorem \ref{mainthm2}.    

\section{Preliminaries}\label{Sec2}
%
%
In this section we shall give some important identities and 
useful estimates.
First we show local existence and first extensibility 
criterion which contains possibility of extinction and 
gradient blow-up of solutions.
\begin{lem}\label{Lem2.1}
Assume that $u_0$ satisfies \eqref{u0}. Then there exist 
$\tmax\in(0,\infty]$ and a pair $(u,v)$ of 
radially symmetric positive functions 
 \begin{align*}
u\in C^{2,1}(\overline{\Omega}\times[0,\tmax))
\quad \mbox{and}\quad
v\in C^{2,0}(\overline{\Omega}\times[0,\tmax))
\end{align*}
which satisfy \eqref{P} in the classical sense in
$\Omega\times(0,\tmax)$. 
Moreover, 
\begin{align}\label{1st-criterion}
\mbox{if}\ \ \tmax<\infty,\ \mbox{then either}
\ 
\liminf_{t\nearrow\tmax}\inf_{x\in\Omega}u(x,t)=0
\ \mbox{or}\ 
\limsup_{t\nearrow\tmax}\|u(\cdot,t)\|_{W^{1,\infty}(\Omega)}
=\infty.
\end{align}
\end{lem}
\begin{proof}
The proof is based on that of
 \cite[Lemma 2.1]{B-W}. Put
\begin{align*}
\ep:=\min\left\{\frac{1}{2}\inf_{x\in\Omega} u_0(x), 
\frac{1}{2\|u_0\|_{L^\infty(\Omega)}}, 
\frac{1}{2\|\nabla u_0\|_{L^\infty(\Omega)}}, 2 \right\}
\end{align*}
and let $\psi_\ep, \varphi_\ep\in C^\infty(\mathbb{R})$ be cut-off 
 functions satisfying 
\begin{align*}
\frac{\ep}{2}\leq \psi_\ep(s) \leq \frac{2}{\ep}\quad \mbox{for all} \ 
s\in \mathbb{R} \qquad \mbox{and} \qquad \psi_\ep(s)=s\quad
\mbox{for all}\  s\in \left(\ep,\frac{1}{\ep}\right)
\end{align*}
as well as
\begin{align*}
\frac{\ep}{2}\leq \varphi_\ep(s) \leq \frac{2}{\ep}
\quad \mbox{for all}\ s\in \mathbb{R} 
\qquad \mbox{and} \qquad 
\varphi_\ep(s)=s\quad
\mbox{for all}\  s\in \left(\ep,\frac{1}{\ep}\right).
\end{align*}
Then we can see that the function 
$a_\ep\in C^\infty(\mathbb{R}\times\mathbb{R}^n)$ defined as 
\begin{align*}
a_\ep(s,\xi) 
:= \frac{\psi^p_\ep(s)}{\sqrt{\psi^2_\ep(s)+\varphi^2_\ep(|\xi|)}},\quad 
s\in \mathbb{R},\  \xi\in \mathbb{R}^n,
\end{align*}
fulfills 
\begin{align*}
\frac{\ep^{p+1}}{2^{p+1}\sqrt{2}}\leq a_\ep(s,\xi)
\leq \left(\frac{2}{\ep}\right)^{p-1}  
\end{align*}
for all $s\in \mathbb{R}$ and all $\xi\in \mathbb{R}^n$. Therefore, applying a fixed point argument enables us to take 
 $T_\ep>0$ and functions 
  \begin{align*}
  u_\ep\in C^{2,1}(\overline{\Omega}\times[0,T_\ep))
\quad \mbox{and}\quad 
v_\ep\in C^{2,0}(\overline{\Omega}\times[0,T_\ep)) 
\end{align*}
 such that $(u_\ep,v_\ep)$ is a classical solution of the problem 
 \begin{equation*}
     \begin{cases}
         u_t = \nabla\cdot\left(a_\ep(s,\nabla u) \nabla u \right) 
         -\chi \nabla\cdot\left(
         \dfrac{u^q\nabla v}{\sqrt{1 + 
         |\nabla v|^2}}\right),
         &x\in \Omega,\ t\in(0,T_\ep),
\\[2mm]
         0 = \Delta v - \mu + u, 
         &x\in \Omega,\ t\in(0,T_\ep),         
\\[2mm]
         \frac{\partial u^p}{\partial\nu}
           = \frac{\partial v}{\partial\nu}=0, 
         &x \in \partial\Omega,\ t\in(0,T_\ep),
\\[2mm]
         u(x,0) = u_0(x),
         &x \in \Omega,
     \end{cases}
 \end{equation*}
and that $u_\ep$ and $v_\ep$ are radially symmetric and positive.
Thus, aided by the argument in the proof of
\cite[Lemma 2.1]{B-W}, 
we can attain this lemma.
\end{proof}
%
%
In the following, 
we assume that $u_0$ satisfies \eqref{u0} and 
 denote by $(u,v)$ and $\tmax$
the local solution of \eqref{P} and the maximal existence time 
 which are obtained in Lemma \ref{Lem2.1}. 
Thanks to the properties that $u$ and $v$ are radially symmetric, 
we can obtain a useful identity of $u_t$. 
By introducing $r:=|x|$ we regard $u(x,t)$ and $v(x,t)$ 
as $u(r,t)$ and $v(r,t)$, respectively.
\begin{lem}\label{ut.lem}
 The solution of \eqref{P} satisfies 
 \begin{align}\label{ut}
 u_t&=\frac{u^{p+2}u_{rr}}{{\sqrt{u^2+u^2_r}}^3}
 +p\frac{u^{p-1}u^4_r}{{\sqrt{u^2+u^2_r}}^3}
 +\frac{n-1}{r}\cdot\frac{u^pu_r}{\sqrt{u^2+u^2_r}}
 +(p-1)\frac{u^{p+1}u^2_r}{{\sqrt{u^2+u^2_r}}^3}
 \\
\nonumber &\quad\,
-q\chi\frac{u^{q-1}u_rv_r}{\sqrt{1+v^2_r}}
-\chi\frac{u^q(\mu-u)}{{\sqrt{1+v^2_r}}^3}
-\chi\frac{n-1}{r}\cdot\frac{u^qv^3_r}{{\sqrt{1+v^2_r}}^3} 
 \end{align}
  for all $ r \in (0, R)$ and all $t \in (0, \tmax)$.
\end{lem}
\begin{proof}
We rewrite \eqref{P} as
\begin{align}\label{ut;first}
 u_t &= \frac{1}{r^{n-1}}\left(r^{n-1}\cdot\frac{u^pu_r}
     {\sqrt{u^2+u^2_r}}\right)_r
     -\chi\frac{1}{r^{n-1}}\left(r^{n-1}\cdot
     \frac{u^qv_r}{\sqrt{1+v^2_r}}\right)_r\\ 
     \nonumber
     &=\frac{n-1}{r}\cdot\frac{u^pu_r}{\sqrt{u^2+u^2_r}}
     +\frac{pu^{p-1}u^2_r+u^pu_{rr}}{\sqrt{u^2+u^2_r}}
     -\frac{u^pu_r(2uu_r+2u_ru_{rr})}{2\,{\sqrt{u^2+u^2_r}}^3} \\ 
     &\quad\, \nonumber
     -\chi\frac{n-1}{r}\cdot\frac{u^qv_r}{\sqrt{1+v^2_r}}
     -\chi\frac{qu^{q-1}u_rv_r+u^qv_{rr}}{\sqrt{1+v^2_r}}
     +\chi\frac{u^qv_r\cdot2v_rv_{rr}}{2\,{\sqrt{1+v^2_r}}^3}
     \end{align}
      for all $ r \in (0, R)$ and all $t \in (0, \tmax)$. 
Here we simplify the second and third terms as
      \begin{align*}
     &\frac{pu^{p-1}u^2_r+u^pu_{rr}}{\sqrt{u^2+u^2_r}}
     -\frac{u^pu_r(2uu_r+2u_ru_{rr})}{2\,{\sqrt{u^2+u^2_r}}^3}\\
   &\qquad=\frac{pu^{p-1}u^2_r}{\sqrt{u^2+u^2_r}} 
   +\frac{u^pu_{rr}}{\sqrt{u^2+u^2_r}}
   -\frac{u^{p+1}u^2_r+u^pu^2_ru_{rr}}
   {\,{\sqrt{u^2+u^2_r}}^3} \\ 
   &\qquad
   =\frac{u^{p+2}u_{rr}}{{\sqrt{u^2+u^2_r}}^3}
   +p\frac{u^{p-1}u^4_r}{{\sqrt{u^2+u^2_r}}^3}
   +(p-1)\frac{u^{p+1}u^2_r}{\,{\sqrt{u^2+u^2_r}}^3}
      \end{align*}
 for all $ r \in (0, R)$ and all $t \in (0, \tmax)$. Similarly we 
simplify the fourth, fifth, sixth terms 
on the right-hand side of \eqref{ut;first} to obtain
\begin{align*}
 -\chi&\frac{n-1}{r}\cdot\frac{u^qv_r}{\sqrt{1+v^2_r}}
      -\frac{qu^{q-1}u_rv_r+u^qv_{rr}}{\sqrt{1+v^2_r}}
     +\frac{u^qv_r^2v_{rr}}{{\sqrt{1+v^2_r}}^3}
    \\   \nonumber
  &=-\chi\frac{n-1}{r}
  \cdot\frac{u^qv_r(1+v^2_r)}{\sqrt{1+v^2_r}^3}
     -q\chi\frac{u^{q-1}u_rv_r}{\sqrt{1+v^2_r}}
     -\chi\frac{u^qv_{rr}(1+v^2_r)}{\sqrt{1+v^2_r}^3}
       +\chi\frac{u^qv_r^2v_{rr}}{{\sqrt{1+v^2_r}}^3}
         \\   \nonumber
  &=-q\chi\frac{u^{q-1}u_rv_r}{\sqrt{1+v^2_r}}
-\chi\frac{n-1}{r}\cdot\frac{u^qv^3_r}{{\sqrt{1+v^2_r}}^3} 
+\chi\frac{u^q}{{\sqrt{1+v^2_r}}^3}\left(v_{rr}+\frac{n-1}{r}v_r\right).
     \end{align*}
Using 
\[
v_{rr}+\frac{n-1}{r}v_r=\mu-u,
\] 
which can be seen 
from the second equation of \eqref{P}, 
we have the conclusion of this lemma.
\end{proof}
%
%
Next we establish a parabolic partial  differential equation 
which is satisfied by $u_{r}$.
In the following lemma we will also introduce important operators 
$\mathcal{P}$ and $\mathcal{Q}$.
\begin{lem}\label{urt.lem}
 The solution of \eqref{P} satisfies 
  \begin{align*}
 u_{rt}
 =A_1(r,t)u_{rrr}+A_2(r,t)u_{rr}+a_3(r,t)u^2_r+A_3(r,t)u_r+A_4(r,t)
 \end{align*}
  for all $ r \in (0, R)$ and all $t \in (0, \tmax)$, 
   \begin{align*}
A_1(r,t)&:=\frac{u^{p+2}}{{\sqrt{u^2+u^2_r}}^3},\\
  A_2(r,t)&:=(p+2)\frac{u^{p+1}u^3_r}{{\sqrt{u^2+u^2_r}}^5}
  -3\frac{u^{p+2}u_ru_{rr}}{{\sqrt{u^2+u^2_r}}^5}
  +(p-1)\frac{u^{p+3}u_r}{{\sqrt{u^2+u^2_r}}^5}
  +4p\frac{u^{p+1}u^3_r}{{\sqrt{u^2+u^2_r}}^5}
  \\
  &\quad\ 
  +p\frac{u^{p-1}u^5_r}{{\sqrt{u^2+u^2_r}}^5}
  +\frac{n-1}{r}\cdot\frac{u^{p+2}}{{\sqrt{u^2+u^2_r}}^3}
  +(p-1)\frac{u^{p+1}u_r}{{\sqrt{u^2+u^2_r}}^5}(2u^2-u^2_r)
  \\
  &\quad\ 
  -q\chi\frac{u^{q-1}u_rv_r}{\sqrt{1+v^2_r}},\\
      a_3(r,t)&:=p(p-1)\frac{u^{p-2}u^5_r}
      {{\sqrt{u^2+u^2_r}}^5}
      -q(q-1)\chi\frac{u^{q-2}v_r}{\sqrt{1+v^2_r}}
      \end{align*}
and
\begin{align*}
A_3(r,t)&:=p(p-4)\frac{u^pu^4_r}{{\sqrt{u^2+u^2_r}}^5}
-\frac{n-1}{r^2}\cdot\frac{u^p}{\sqrt{u^2+u^2_r}}+\Phi(r,t),\\
A_4(r,t)&:=p\frac{n-1}{r}\cdot\frac{u^{p-1}u^4_r}
{{\sqrt{u^2+u^2_r}}^3}+\Psi(r,t)
\end{align*}
as well as
\begin{align*}
   \Phi(r,t)&:=(p-1)(p-2)\frac{u^{p+2}u_r}{{\sqrt{u^2+u^2_r}}^5}
         +(p-1)(p+1)\frac{u^pu^3_r}{{\sqrt{u^2+u^2_r}}^5}
         -q\chi\mu\frac{u^{q-1}}{{\sqrt{1+v^2_r}}^3}
         \\
         &\quad\ 
         +(q+1)\chi\frac{u^q}{{\sqrt{1+v^2_r}}^3}
         -q\chi\frac{u^{q-1}v_{rr}}{{\sqrt{1+v^2_r}}^3}
         +q\chi\frac{u^{q-1}v^2_rv_{rr}}{{\sqrt{1+v^2_r}}^3}
         -q\chi\frac{n-1}{r}\cdot\frac{u^{q-1}v^3_r}
         {{\sqrt{1+v^2_r}}^3},
\end{align*}
and
\begin{align*}
  \Psi(r,t)&:=(p-1)\frac{n-1}{r}\cdot\frac{u^{p+1}u^2_r}
{{\sqrt{u^2+u^2_r}}^3}
+3\chi\mu\frac{u^qv_rv_{rr}}{{\sqrt{1+v^2_r}}^5}
-3\chi\frac{u^{q+1}v_rv_{rr}}{{\sqrt{1+v^2_r}}^5} 
\\
&\quad\ 
+\chi\frac{n-1}{r^2}\cdot\frac{u^qv^3_r}{{\sqrt{1+v^2_r}}^3}
-3\chi\frac{n-1}{r}\cdot\frac{u^qv^2_rv_{rr}}{{\sqrt{1+v^2_r}}^5}
   \end{align*}
  for $ r \in (0, R)$ and $t \in (0, \tmax)$. 
  In particular, 
  \begin{align*}
   (\mathcal{P}u_r)(r,t)=0
   \end{align*}
for all $r \in (0,R)$ and all $t \in (0, \tmax)$, 
with $\mathcal{P}$ given by 
   \begin{align}\label{pphi}
   (\mathcal{P}\varphi)(r,t)
   :={\varphi}_t-A_1(r,t){\varphi}_{rr}-A_2(r,t){\varphi}_r
   -a_3(r,t){\varphi}^2-A_3(r,t)\varphi-A_4(r,t)
   \end{align}
   for $r \in (0, R)$ and $t \in (0, \tmax)$. Likewise,
   \begin{align*}
   (\mathcal{Q}u_r)(r,t)=0
   \end{align*}
for all $r \in (0,R)$ and all $t \in (0, \tmax)$, 
 with $\mathcal{Q}$ given by
   \begin{align}\label{qphi}
   (\mathcal{Q}\varphi)(r,t)
   :={\varphi}_t-A_1(r,t){\varphi}_{rr}-A_2(r,t){\varphi}_r
   -a_3(r,t){\varphi}^2
   -\widetilde{A}_3(r,t)\varphi-\widetilde{A}_4(r,t)
   \end{align}
 for $ r \in (0, R)$ and $t \in (0, \tmax)$,
 where
 \begin{align}\label{Atilde}
 \widetilde{A}_3(r,t)
 &:=p\frac{n-1}{r}\cdot\frac{u^{p-1}u^3_r}
 {{\sqrt{u^2+u^2_r}}^3}+\Phi(r,t),
 \\ \notag
\widetilde{A}_4(r,t)&:=p(p-4)\frac{u^pu^4_r}
{{\sqrt{u^2+u^2_r}}^5}
-\frac{n-1}{r^2}\cdot\frac{u^p}{\sqrt{u^2+u^2_r}}+\Psi(r,t)
 \end{align}
  for $ r \in (0, R)$ and $t \in (0, \tmax)$.
\end{lem}
\begin{proof}
We first calculate a partial derivative of \eqref{ut} 
with respect to $r$ to obtain
  \begin{align}\label{urt}
 u_{rt}&=\frac{u^{p+2}u_{rrr}}{{\sqrt{u^2+u^2_r}}^3}
 +(p+2)\frac{u^{p+1}u_ru_{rr}}{{\sqrt{u^2+u^2_r}}^3}
  -\frac{3}{2}\cdot\frac{u^{p+2}u_{rr}(2uu_r+2u_ru_{rr})}
  {{\sqrt{u^2+u^2_r}}^5}\\
\nonumber &\quad\,
 +4p\frac{u^{p-1}u^3_ru_{rr}}{{\sqrt{u^2+u^2_r}}^3}
 +p(p-1)\frac{u^{p-2}u^5_r}{{\sqrt{u^2+u^2_r}}^3}
 -\frac{3}{2}p\cdot\frac{u^{p-1}u^4_r(2uu_r+2u_ru_{rr})}
 {{\sqrt{u^2+u^2_r}}^5}\\
\nonumber &\quad\,
 -\frac{n-1}{r^2}\cdot\frac{u^pu_r}{\sqrt{u^2+u^2_r}}
 +\frac{n-1}{r}\cdot\frac{u^pu_{rr}}{{\sqrt{u^2+u^2_r}}}
 +p\frac{n-1}{r}\cdot\frac{u^{p-1}u^2_r}{\sqrt{u^2+u^2_r}}
 \\
\nonumber &\quad\,
 -\frac{1}{2}\cdot\frac{n-1}{r}\cdot\frac{u^pu_r(2uu_r+2u_ru_{rr})}
 {{\sqrt{u^2+u^2_r}}^3}
 +(p-1)(p+1)\frac{u^pu^3_r}{{\sqrt{u^2+u^2_r}}^3}
 \\
 \nonumber &\quad\,
 +2(p-1)\frac{u^{p+1}u_ru_{rr}}{{\sqrt{u^2+u^2_r}}^3}
  -\frac{3}{2}(p-1)\frac{u^{p+1}u^2_r(2uu_r+2u_ru_{rr})}
  {{\sqrt{u^2+u^2_r}}^5}
  \\
\nonumber &\quad\,
 -q\chi\mu\frac{u^{q-1}u_r}{{\sqrt{1+v^2_r}}^3}
 +(q+1)\chi\frac{u^qu_r}{{\sqrt{1+v^2_r}}^3}
 +\frac{3}{2}\chi\frac{u^q(\mu-u)\cdot2v_rv_{rr}}
 {{\sqrt{1+v^2_r}}^5}
 \\
\nonumber &\quad\,
 -q(q-1)\chi\frac{u^{q-2}u^2_rv_r}{\sqrt{1+v^2_r}}
 -q\chi\frac{u^{q-1}u_rv_ru_{rr}}{\sqrt{1+v^2_r}}
 -q\chi\frac{u^{q-1}u_rv_{rr}}{\sqrt{1+v^2_r}}
 \\
\nonumber &\quad\,
 +\frac{1}{2}\cdot q\chi\frac{u^{q-1}u_rv_r\cdot2v_rv_{rr}}
 {{\sqrt{1+v^2_r}}^3}
 +\chi\frac{n-1}{r^2}\cdot\frac{u^qv^3_r}
 {{\sqrt{1+v^2_r}}^3}
 -q\chi\frac{n-1}{r}\cdot\frac{u^{q-1}u_rv^3_r}
 {{\sqrt{1+v^2_r}}^3}
 \\
\nonumber &\quad\,
 -3\chi\frac{n-1}{r}\cdot
 \frac{u^qv^2_rv_{rr}}{{\sqrt{1+v^2_r}}^3}
 +\frac{3}{2}\chi\frac{n-1}{r}\cdot
 \frac{u^qv^3_r\cdot2v_rv_{rr}}{{\sqrt{1+v^2_r}}^5}
 \end{align}
 for all $ r \in (0, R)$ and all $t \in (0, \tmax)$. 
   By simplifying the fourth, fifth and sixth terms
    on the right-hand side of \eqref{urt} according to
 \begin{align*}
  4p&\frac{u^{p-1}u^3_ru_{rr}}{{\sqrt{u^2+u^2_r}}^3}
     +p(p-1)\frac{u^{p-2}u^5_r}{{\sqrt{u^2+u^2_r}}^3}
        -\frac{3}{2}p\frac{u^{p-1}u^4_r(2uu_r+2u_ru_{rr})}
        {{\sqrt{u^2+u^2_r}}^5}\\
     &=4p\frac{u^{p+1}u^3_ru_{rr}}{{\sqrt{u^2+u^2_r}}^5}
       +4p\frac{u^{p-1}u^5_ru_{rr}}{{\sqrt{u^2+u^2_r}}^5}
       +p(p-1)\frac{u^pu^5_r}{{\sqrt{u^2+u^2_r}}^5}
         \\
       &\quad\,
       +p(p-1)\frac{u^{p-2}u^7_r}{{\sqrt{u^2+u^2_r}}^5}
       -3p\frac{u^pu^5_r}{{\sqrt{u^2+u^2_r}}^5}
       -3p\frac{u^{p-1}u^5_ru_{rr}}
       {{\sqrt{u^2+u^2_r}}^5}\\
    &=4p\frac{u^{p+1}u^3_ru_{rr}}{{\sqrt{u^2+u^2_r}}^5}
       +p\frac{u^{p-1}u^5_ru_{rr}}{{\sqrt{u^2+u^2_r}}^5}
       +p(p-4)\frac{u^pu^5_r}{{\sqrt{u^2+u^2_r}}^5}
       +p(p-1)\frac{u^{p-2}u^7_r}{{\sqrt{u^2+u^2_r}}^5},
  \end{align*}
 arguments similar to those in the proof of 
 \cite[Lemma \ref{urt.lem}]{B-W} entail this lemma. 
  \end{proof}
\begin{remark}
In the proof of Lemma \ref{urt.lem},
 the difference between our study and \cite{B-W} is the fact that there exist new terms  
 $p(p-1)\frac{u^{p-2}u^7_r} {{\sqrt{u^2+u^2_r}}^5}
                  -q(q-1)\chi\frac{u^{q-2}u^2_rv_r}{\sqrt{1+v^2_r}}$ 
which do not exist in the case that $p=q=1$. 
Then we will control these terms by introducing $a_3(r,t)u^2_r$.
The rest terms in \eqref{urt} are adequately distributed between 
$A_3(r,t)$ and  $A_4(r,t)$, as well as $\widetilde{A_3}(r,t)$ and  
$\widetilde{A_4}(r,t)$.
\end{remark}
The following lemmas are utilized to establish 
useful estimates for $v$. Since the proofs of these lemmas are in 
\cite[Lemmas 2.4, 2.5]{B-W}, we provide only
 the statements of lemmas.
%
%
\begin{lem}\label{lem2.4}
 Assume that $u_0$ satisfies \eqref{u0}. Then
 \begin{align*}
  v_r(r,t)=\frac{{\mu}r}{n}-r^{1-n}\cdot\int_{0}^{r}
                \rho^{n-1}u(\rho,t)\,d\rho  
 \end{align*}
 and 
 \begin{align*}
   v_{rr}(r,t)=\frac{{\mu}}{n}-u
                   +\frac{n-1}{r^n}\cdot\int_{0}^{r}
                      \rho^{n-1}u(\rho,t)\,d\rho  
 \end{align*}
 for all $ r \in (0, R)$ and all $t \in (0, \tmax)$. Moreover, we have
 \begin{align*}
  v_{rt}(r,t)=-\frac{u^pu_r}{\sqrt{u^2+u^2_r}}
             +\chi\frac{u^qv_r}{\sqrt{1+v^2_r}}
 \end{align*}
 for all $ r \in (0, R)$ and all $t \in (0, \tmax)$.
\end{lem}
%
%
\begin{lem}\label{lem2.5}
 Let $u_0$ satisfy \eqref{u0}. Then for 
 all $ r \in (0,R)$ and all $t \in (0, \tmax)$,
  we have 
 \begin{align*}
 -\frac{{\mu}R^n}{n}\cdot r^{1-n}\leq v_r(r,t)
             \leq\frac{\mu}{n}\cdot r
 \end{align*}
 and
 \begin{align*}
|v_r(r,t)|\leq\frac{\|u(\cdot,t)\|_{L^{\infty}(0,R)}}{n}\cdot r
 \end{align*}
 as well as
 \begin{align*}
|v_{rr}(r,t)|\leq\|u(\cdot,t)\|_{L^{\infty}(0,R)}.
 \end{align*}
\end{lem}
%
\section{A pointwise lower estimate for $u$}\label{Sec3}
In this section we will rule out the possibility of 
$\liminf_{t\nearrow\tmax}\inf_{x\in\Omega}u(x,t)=0$
in \eqref{1st-criterion}. 
In order to attain this purpose we show the following lower 
estimate for $u$.
 \begin{lem}\label{lem;rulingout-extinction}
  Assume that $\tmax<\infty$, 
  but that
  $\sup_{(r,t)\in(0,R)\times(0, \tmax)}u(r,t)<\infty.$ Then 
    \begin{equation}\label{ulb}
    u(r,t)\geq\left(\inf_{s\in(0,R)} u_0(s)\right) e^{-\kappa t}
    \end{equation}
for all $ r \in (0, R)$ and all $t \in (0, \tmax)$, where 
    \begin{equation}\label{kappa}
    \kappa:=2\chi\mu\|u\|^{q-1}
    _{L^{\infty}((0,R)\times(0,\tmax))}.
    \end{equation}
 \end{lem}
\begin{proof}
 We rewrite \eqref{ut} as
  \begin{equation}\label{uta}
  u_t=a_1(r,t)u_{rr}+a_{21}(r,t)u_r
 +\frac{a_{22}(r,t)}{r}\cdot u_r
-\chi\frac{u^q(\mu-u)}{{\sqrt{1+v^2_r}}^3}
-\chi\frac{n-1}{r}\cdot\frac{u^qv^3_r}
{{\sqrt{1+v^2_r}}^3} 
  \end{equation}
  for all $ r \in (0, R)$ and all $t \in (0, \tmax)$, where
 \begin{align*}
   a_1(r,t)&:=\frac{u^{p+2}}{{\sqrt{u^2+u^2_r}}^3},\\
   a_{21}(r,t)&:=p\frac{u^{p-1}u^3_r}{{\sqrt{u^2+u^2_r}}^3}
 +(p-1)\frac{u^{p+1}u_r}{{\sqrt{u^2+u^2_r}}^3}
-q\chi\frac{u^{q-1}v_r}{\sqrt{1+v^2_r}},\\
   a_{22}(r,t)&:=(n-1)\frac{u^p}{\sqrt{u^2+u^2_r}}
   \end{align*}
are continuous functions in $[0,R]\times(0,\tmax)$. 
By using the boundedness of $u$, we can establish that
\begin{align}\label{3esti1}
   -\chi\frac{u^q(\mu-u)}{{\sqrt{1+v^2_r}}^3}
   \geq-\chi\mu\frac{u^q}{\sqrt{1+v^2_r}^3}
   \geq-\chi\mu\cdot\|u\|^{q-1}
   _{L^{\infty}((0,R)\times(0,\tmax))}\cdot u
   \end{align}
    for all $ r \in (0, R)$ and all $t \in (0, \tmax)$.
   Moreover, we use the one-sided inequality $v_r\leq\frac{\mu r}{n}$
   provided by Lemma \ref{lem2.5} to obtain 
   \begin{align}\label{3esti2}
   -\chi\frac{n-1}{r}\cdot\frac{u^qv^3_r}
   {{\sqrt{1+v^2_r}}^3}
   &\geq-(n-1)\chi\cdot\frac{v^2_r}
   {{\sqrt{1+v^2_r}}^3}\cdot\frac{v_r}{r}\cdot u^q\\ \nonumber
   &\geq-(n-1)\chi\cdot 1\cdot\frac{\mu}{n}\cdot
   \|u\|^{q-1}_{L^{\infty}((0,R)\times(0,\tmax))}\cdot u\\ \nonumber
   &\geq-\chi\mu\cdot\|u\|^{q-1}
   _{L^{\infty}((0,R)\times(0,\tmax))}\cdot u
   \end{align}
    for all $ r \in (0, R)$ and all $t \in (0, \tmax)$.
    Thus plugging \eqref{3esti1} and \eqref{3esti2} into \eqref{uta} 
    implies 
     \begin{align*}
      u_t\geq a_1(r,t)u_{rr}+a_{21}(r,t)u_r
 +\frac{a_{22}(r,t)}{r }\cdot u_r
-\kappa u
     \end{align*}
  for all $ r \in (0, R)$ and all $t \in (0, \tmax)$ 
  with $\kappa$ as in \eqref{kappa}. 
Thanks to the contradiction arguments similar to those in the proof of  \cite[Lemma 3.2]{B-W}, we arrive at the conclusion. 
\end{proof}
%
\section{A pointwise lower estimate for $u_r$}\label{Sec4}
%
%
%
In this section we will establish a key estimate. 
We confirm the following lemma that not only implies a 
lower bound for $u_r$ but also will play an important role 
to obtain an upper estimate for $u_r$. 
\begin{lem}\label{lem4.1}
  Assume that $\tmax<\infty$, 
  but that $\sup_{(r,t)\in(0,R)\times(0, \tmax)}u(r,t)<\infty.$ 
  Then there exists a constant $C>0$ such that
    \begin{align*}
    u_r(r,t)\geq-C
    \end{align*}    
    for all $ r \in (0, R)$ and all $t \in (0, \tmax)$.
 \end{lem}
\begin{proof}
From the assumption of this lemma  
 we can find a constant $c_1>0$ such that
\begin{align}\label{4u}
 u(r,t)\leq c_1\quad
  \mbox{for all}\ r \in (0,R)\ \mbox{and all}\ t \in (0, \tmax),
\end{align} 
which implies that Lemma \ref{lem2.5} provides constants 
$c_2>0$ and $c_3>0$ such that
 \begin{align}\label{4vrvrr}
|v_r(r,t)|\leq c_2r\quad \mbox{and} \quad|v_{rr}(r,t)|\leq c_3
 \end{align}
for all $ r \in (0, R)$ and all $t \in (0, \tmax)$. We now take $\widetilde{c_4}>0$ and $\widetilde{c_5}>0$ fulfilling that
\begin{align*}
\widetilde{c_4}>c_4&
:=p(p-1)c^{p-2}_1+q(q-1)\chi c^{q-2}_1,\\
\widetilde{c_5}>c_5&
:=3p(p+1)c^{p-2}_1+q(c_1+2c_3+\mu)\chi c^{q-1}_1
\end{align*}
and 
 \begin{align}\label{4c4c6c5}
4\widetilde{c_4}c_6-{\widetilde{c_5}}^2<0,
\end{align}
where
\begin{align*}
c_6:=3\left(\mu c_3+c_1c_3
+\frac{(n-1)c^2_2}{3}+(n-1)c_2c_3\right)\chi c^q_1c_2R.
\end{align*}
Then there exist $n\in\mathbb{N}$ and $j\in\{0,1,2,3,4,5\}$
 such that
\begin{align*}
  2\cdot\frac{(6n-6+j)\pi}{6\cdot\frac{2}{3}}
  <\tmax<2\cdot\frac{(6n-5+j)\pi}{6\cdot\frac{2}{3}}.
\end{align*}
Therefore we can find $\ep>0$ such that
\begin{align*}
 2\cdot\frac{(6n-6+j)\pi}{6\cdot\frac{2}{3}}
 <\tmax-\ep<\tmax
 <2\cdot\frac{(6n-5+j)\pi}{6\cdot\frac{2}{3}},
\end{align*}
and then there exists $\alpha_0>0$ such that  
\begin{align}\label{talpha}
 2\cdot\frac{(6n-6+j)\pi}{6\cdot(\frac{2}{3}+\alpha)}
 <\tmax-\ep<\tmax
 <2\cdot\frac{(6n-5+j)\pi}{6\cdot(\frac{2}{3}+\alpha)}
\end{align}
for all $\alpha\in(0,\alpha_0)$. Now we take $E\geq1$ fulfilling
\begin{align*}
 u_r(r,\tmax-\ep)\geq-E 
\end{align*} 
    for all $ r \in (0, R)$.
Then, since $x\tan{\frac{-\pi}{3x}}\to-\infty$ 
as $x\searrow\frac{2}{3}$, 
 we can find $\alpha_1\in(0,\alpha_0)$ such that 
\begin{align*}
 -E>\frac{1}{2\widetilde{c_4}}\left(\frac{2}{3}+\alpha_1\right)
  \tan{\frac{-\pi}{3\cdot\left(\frac{2}{3}+\alpha_1\right)}}
 +\frac{\widetilde{c_5}}{2\widetilde{c_4}}.
\end{align*}
By virtue of \eqref{4c4c6c5} and the fact that 
$4\widetilde{c_4}x-{\widetilde{c_5}}^2\to\infty$ as $x\to\infty$,  
we obtain from the intermediate value theorem that
there is a constant $\widetilde{c_6}>c_6$ such that
\begin{align}\label{alpha1}
 \sqrt{4\widetilde{c_4}\widetilde{c_6}
 -{\widetilde{c_5}}^2}=\frac{2}{3}+\alpha_1.
\end{align}
Combination of \eqref{talpha} and \eqref{alpha1} with 
$\alpha=\alpha_1$ implies that 
\begin{align*}
 \frac{(6n-6+j)\pi}{6\cdot\sqrt{\frac{4\widetilde{c_4}\widetilde{c_6}
 -{\widetilde{c_5}}^2}{4}}}
 <\tmax-\ep<\tmax<
 \frac{(6n-5+j)\pi}{6\cdot\sqrt{\frac{4\widetilde{c_4}\widetilde{c_6}
 -{\widetilde{c_5}}^2}{4}}}.
\end{align*}
Now we define a comparison function $\underline{\varphi}$    
   (see Figure 1) by letting
  \begin{equation*}
  \underline{\varphi}(r,t):=    
     \begin{cases}
        D+\frac{\widetilde{c_5}}{2\widetilde{c_4}},
        &t=0,
\\[4mm] 
         \sqrt{\widetilde{C}}
          \tan\left[\tan^{-1}\frac{D}{\sqrt{\widetilde{C}}}
          -\widetilde{c_4}\sqrt{\widetilde{C}}\left(t-\frac{j\pi}
          {6\widetilde{c_4}\sqrt{\widetilde{C}}}\right)\right]
          +\frac{\widetilde{c_5}}{2\widetilde{c_4}},
          &t\in\left(\frac{(6n-6+j)\pi}
          {6\widetilde{c_4}\sqrt{\widetilde{C}}},
          \frac{(6n-5+j)\pi}{6\widetilde{c_4}
          \sqrt{\widetilde{C}}}\right]
     \end{cases}
 \end{equation*}
 for $r\in[0,R]$, $n\in\mathbb{N}$ and $j\in\{0,1,2,3,4,5\}$, where
 \begin{align*}
 \widetilde{C}
 :=\frac{4\widetilde{c_4}\widetilde{c_6}-{\widetilde{c_5}}^2}
 {4{\widetilde{c_4}}^2}\quad\mbox{and}\quad 
 D:= \frac{1}{2\widetilde{c_4}}\left(\frac{2}{3}+\alpha_1\right)
        \tan{\frac{-\pi}
        {3\cdot\left(\frac{2}{3}+\alpha_1\right)}}.
 \end{align*}
 \begin{figure}[h]
\centering
\includegraphics[width=0.47\textwidth]{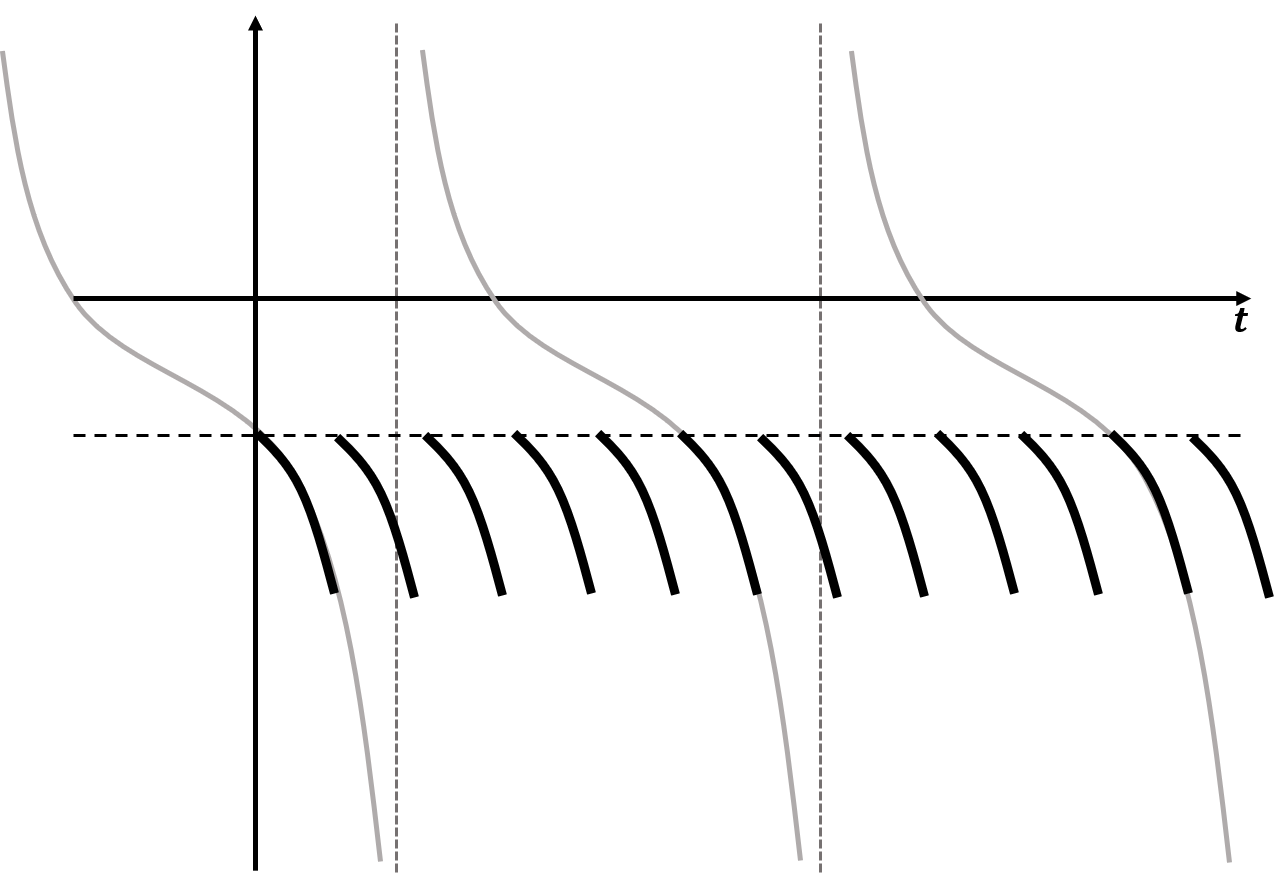}
\caption{Graph of the function $\underline{\varphi}$}
\end{figure}%

\noindent
Our goal is to show that $u_r\geq \underline{\varphi}$. 
Toward this goal, in view of the comparison principle, it is enough to verify that $\mathcal{P}\underline{\varphi}\leq 0$. 
Since 
$\underline{\varphi}$ is a monotonically decreasing function 
with respect to $t
\in \left(\frac{(6n-6+j)\pi}{6\widetilde{c_4}\sqrt{\widetilde{C}}},
        \frac{(6n-5+j)\pi}{6\widetilde{c_4}\sqrt{\widetilde{C}}}\right]$ 
     for all $n\in\mathbb{N}$ and all $j\in\{0,1,2,3,4,5\}$,
it follows that 
\[ 
 \underline{\varphi}(r,t)\leq \underline{\varphi}(r,0)<-E
\] 
for all $ r \in (0, R)$ and all $t \in (0, \tmax)$, 
and hence 
\[\underline{\varphi}<0.
\] 
Noting that 
\[ \underline{\varphi}_r= \underline{\varphi}_{rr}=0 
\quad \mbox{and} \quad  
 -\underline{\varphi}=|\underline{\varphi}|
\] 
because $\underline{\varphi}$ is independent of
$r$ and $\underline{\varphi}<0$, 
     we obtain from \eqref{pphi} that
  \begin{align*}
   (\mathcal{P}\underline{\varphi})(r,t)
   &={\underline{\varphi}}_t-a_3(r,t){\underline{\varphi}}^2
   +A_3(r,t)|\underline{\varphi}|-A_4(r,t)\\
   &={\underline{\varphi}}_t
   -\left(p(p-1)\frac{u^{p-2}u^5_r}{{\sqrt{u^2+u^2_r}}^5}
-q(q-1)\chi\frac{u^{q-2}v_r}
{\sqrt{1+v^2_r}}\right){\underline{\varphi}}^2\\
&\quad\,
   +p(p-4)\frac{u^pu^4_r}
   {{\sqrt{u^2+u^2_r}}^5}|\underline{\varphi}|
  -\frac{n-1}{r^2}\cdot\frac{u^p}
  {\sqrt{u^2+u^2_r}}|\underline{\varphi}|\\
&\quad\,
+(p-1)(p-2)\frac{u^{p+2}u_r}
{{\sqrt{u^2+u^2_r}}^5}|\underline{\varphi}|
+(p-1)(p+1)\frac{u^pu^3_r}
{{\sqrt{u^2+u^2_r}}^5}|\underline{\varphi}|\\
&\quad\,
-q\chi\mu\frac{u^{q-1}}
{{\sqrt{1+v^2_r}}^3}|\underline{\varphi}|
+(q+1)\chi\frac{u^q}
{{\sqrt{1+v^2_r}}^3}|\underline{\varphi}|
-q\chi\frac{u^{q-1}v_{rr}}
{{\sqrt{1+v^2_r}}^3}|\underline{\varphi}|\\
&\quad\,
+q\chi\frac{u^{q-1}v^2_rv_{rr}}
{{\sqrt{1+v^2_r}}^3}|\underline{\varphi}|
-q\chi\frac{n-1}{r}\cdot\frac{u^{q-1}v^3_r}
{{\sqrt{1+v^2_r}}^3}|\underline{\varphi}|\\
&\quad\,
-p\frac{n-1}{r}\cdot\frac{u^{p-1}u^4_r}
{{\sqrt{u^2+u^2_r}}^3}
-(p-1)\frac{n-1}{r}\cdot\frac{u^{p+1}u^2_r}
{{\sqrt{u^2+u^2_r}}^3}
-3\chi\mu\frac{u^qv_rv_{rr}}{{\sqrt{1+v^2_r}}^5}
\\
&\quad\,
+3\chi\frac{u^{q+1}v_rv_{rr}}{{\sqrt{1+v^2_r}}^5}
-\chi\frac{n-1}{r^2}\cdot\frac{u^qv^3_r}
{{\sqrt{1+v^2_r}}^3}
+3\chi\frac{n-1}{r}\cdot\frac{u^qv^2_rv_{rr}}
{{\sqrt{1+v^2_r}}^5}
  \end{align*}
  for all $ r \in (0, R)$ and all $t \in (0, \tmax)$.
 Then, since $\underline{\varphi}$ satisfies
\begin{align*}
{\underline{\varphi}}_t
   =-\widetilde{c_4}{\underline{\varphi}}^2
   +\widetilde{c_5}\underline{\varphi}-\widetilde{c_6}
\end{align*}   
   and the fourth, seventh, twelfth, thirteenth terms
  on the right-hand side are nonpositive:
  \begin{align*}
 & -\frac{n-1}{r^2}\cdot\frac{u^p}
  {\sqrt{u^2+u^2_r}}|\underline{\varphi}|\leq 0,\qquad  
  -q\chi\mu\frac{u^{q-1}}
{{\sqrt{1+v^2_r}}^3}|\underline{\varphi}|\leq 0,\\
&-p\frac{n-1}{r}\cdot\frac{u^{p-1}u^4_r}
{{\sqrt{u^2+u^2_r}}^3}\leq 0, \qquad
-(p-1)\frac{n-1}{r}\cdot\frac{u^{p+1}u^2_r}
{{\sqrt{u^2+u^2_r}}^3}\leq 0, 
  \end{align*}
  we can obtain from \eqref{4u}, \eqref{4vrvrr} and 
  the inequality e.g. 
  $\frac{u^5_r}{\sqrt{u^2+u^2_r}^5}\leq 1$
  that
  \begin{align*}
   (\mathcal{P}\underline{\varphi})(r,t)
   &\leq{\underline{\varphi}}_t
   +\left(p(p-1)c^{p-2}_1
   +q(q-1)\chi c^{q-2}_1\right){\underline{\varphi}}^2\\
&\quad\,
   +\left(3p(p+1)c^{p-2}_1+q(c_1+2c_3+\mu)
   \chi c^{q-1}_1\right)|\underline{\varphi}|\\
&\quad\,
   +3\left(\mu c_3+c_1c_3+\frac{(n-1)c^2_2}{3}
   +(n-1)c_2c_3\right)\chi c^q_1c_2R\\[2mm]
   &=-\widetilde{c_4}{\underline{\varphi}}^2
   +\widetilde{c_5}\underline{\varphi}-\widetilde{c_6}
   +c_4{\underline{\varphi}}^2
   +c_5|\underline{\varphi}|+c_6\\[2mm]
   &=(c_4-\widetilde{c_4}){\underline{\varphi}}^2
   +(c_5-\widetilde{c_5})|\underline{\varphi}|
   +(c_6-\widetilde{c_6})
   \end{align*}
   for all $ r \in (0, R)$ and all $t \in (\tmax-\ep, \tmax)$.
   Then the relations that 
   $\widetilde{c_4}>c_4$, $\widetilde{c_5}>c_5$ 
   and $\widetilde{c_6}>c_6$ 
   ensure that 
   \[(\mathcal{P}\underline{\varphi})(r,t)\leq 0\]
   for all $r\in(0,R)$ and all $t\in(\tmax-\ep, \tmax)$.
   Since \[(\mathcal{P}u_r)(r,t)=0\]
    for all $(r,t)\in(0,R)\times(\tmax-\ep, \tmax)$,
   and moreover
   \begin{align*}
   \underline{\varphi}(r,\tmax-\ep)
      &<\underline{\varphi}(r,0)\\
      &=\frac{1}{2\widetilde{c_4}}
      \left(\frac{2}{3}+\alpha_1\right)
      \tan{\frac{-\pi}{3\cdot
      \left(\frac{2}{3}+\alpha_1\right)}}
      +\frac{\widetilde{c_5}}{2\widetilde{c_4}}\\
      &<-E\\
      &\leq u_r(r,\tmax-\ep)
\end{align*}
for all $r\in [0, R]$
  and
 \begin{align*}
 &\underline{\varphi}(0,t)\leq u_r(0,t)=0,
  \\
   &\underline{\varphi}(R,t)\leq u_r(R,t)=0  
\end{align*}  
for all $t\in(\tmax-\ep, \tmax)$, 
the comparison principle derives that
 $u_r(r,t)\geq \underline{\varphi}(r,t)$
 for all $ r \in (0, R)$ and all $t \in (\tmax-\ep, \tmax)$,
 and hence
  \begin{align*}
  u_r(r,t)&\geq \underline{\varphi}(r,\tmax)\\ 
  &=-| \underline{\varphi}(r,\tmax)|
\end{align*} 
for all $t\in(\tmax-\ep, \tmax)$.
Therefore by putting
\begin{align*}
  C:=  \max\left\{|\underline{\varphi}(r,\tmax)|,\, 
  \max_{(r,t)\in(0,R)\times(0,\tmax-\ep)}{|u_r(r,t)|}\right\}
\end{align*}
we have this lemma.
\end{proof}
%
\section{A bound for $|u_r|$.\ Proof of Theorem 1.1}\label{Sec5}
\subsection{ A bound for $|u_r|$ in terms of\ $z_+$}
Thanks to Lemma \ref{lem4.1}, in order to rule out 
 the possibility of gradient blow-up, it is enough to see that $u_r\leq C$ with some $C>0$. 
Here, in view of arguments as in \cite[Section 5]{B-W}, we will 
 establish a bound for $u_r$ in terms of $z_+$. First we rewrite 
 \eqref{ut} in Lemma \ref{ut.lem} by multiplying $\frac{1}{u}$ on the 
 both sides and find a key quantity such that
\begin{align}\label{z}
 z:=\frac{u_t}{u}&=\frac{u^{p+1}u_{rr}}{{\sqrt{u^2+u^2_r}}^3}
 +p\frac{u^{p-2}u^4_r}{{\sqrt{u^2+u^2_r}}^3}
 +\frac{n-1}{r}\cdot\frac{u^{p-1}u_r}{\sqrt{u^2+u^2_r}}
 +(p-1)\frac{u^pu^2_r}{{\sqrt{u^2+u^2_r}}^3}
 \\
\nonumber &\quad\,
-q\chi\frac{u^{q-2}u_rv_r}{\sqrt{1+v^2_r}}
-\chi\frac{u^{q-1}(\mu-u)}{{\sqrt{1+v^2_r}}^3}
-\chi\frac{n-1}{r}\cdot\frac{u^{q-1}v^3_r}
{{\sqrt{1+v^2_r}}^3}. 
 \end{align}
  This plays an important role when we establish an estimate 
  for $u_r$ which derives the desired extensibility criterion.
%
%
%
%
\begin{lem}\label{lem;ur;origin}
  Assume that $\tmax<\infty$,\,
  but that $\sup_{(r,t)\in(0,R)\times(0, \tmax)}u(r,t)<\infty.$\,
  Then there exist $R_0\in(0,R)$ and a constant $C>0$ such that
    \begin{align*}
    {\|u_r(\cdot,t)\|}_{L^{\infty}(0,R_0)}\leq C
    \left(1+{\|z_+(\cdot,t)\|}_{L^{\infty}(0,R_0)}\right)
    \end{align*}
     for all $t \in (0, \tmax)$.
 \end{lem}
\begin{proof}
The proof is based on an argument in the proof of
 \cite[Lemma 5.1]{B-W}. 
We rewrite \eqref{z} to have
\begin{align}\label{pur4u3}
  p\frac{u^4_r}{u^3}&=\frac{{\sqrt{u^2+u^2_r}}^3}{u^{p+1}}z
 -u_{rr}-\frac{n-1}{r}u_r\cdot\frac{{\sqrt{u^2+u^2_r}}^2}{u^2}
 -(p-1)\frac{u^2_r}{u}
 \\
\nonumber &\quad\,
+q\chi\frac{{\sqrt{u^2+u^2_r}}^3\cdot u_rv_r}
{u^{p-q+3}\sqrt{1+v^2_r}}
+\chi\frac{{\sqrt{u^2+u^2_r}}^3(\mu-u)}
{u^{p-q+2}{\sqrt{1+v^2_r}}^3}
 \\
\nonumber &\quad\,
+\chi\frac{n-1}{r}\cdot
\frac{{\sqrt{u^2+u^2_r}}^3v^3_r}
{u^{p-q+2}{\sqrt{1+v^2_r}}^3},
 \end{align}
 and we have from 
 the identity
 $u_{rr}+\frac{n-1}{r}u_r
 =\frac{1}{r^{n-1}}(r^{n-1}u_r)_r$ that
\begin{align}\label{rapu}
-u_{rr}-\frac{n-1}{r}u_r\cdot\frac{{\sqrt{u^2+u^2_r}}^2}{u^2}
&=-\left(u_{rr}+\frac{n-1}{r}u_r\right)
-\frac{n-1}{r}\cdot\frac{u^3_r}{u^2}\\
\nonumber
&=-\frac{1}{r^{n-1}}(r^{n-1}u_r)_r 
-\frac{n-1}{r}\cdot\frac{u^3_r}{u^2}
\end{align}
  for all $ r \in (0, R)$ and all $t \in (0, \tmax)$. 
   Thanks to the assumption of the boundedness of $u$, 
  we can take constants $c_1\geq\mu$ and $c_2>0$ such that
  \begin{align*}
  u\leq c_1 
\quad \mbox{and} \quad 
  \sqrt{u^2+u^2_r}^3\leq c_2(1+{|u_r|}^3)
\end{align*}
 for all $ r \in (0, R)$ and all $t \in (0, \tmax)$.
 On the other hand, recalling Lemma \ref{lem;rulingout-extinction}, 
  we can find
a constant $c_3>0$ fulfilling
\begin{align}\label{u.ulb}
   u\geq c_3 \qquad
\end{align}
 for all $ r \in (0, R)$ and all $t \in (0, \tmax)$.
Because of the estimate $c_3\leq u \leq c_1$
 for all $ r \in (0, R)$ and all $t \in (0, \tmax)$, 
we can obtain constants $C(p,q)>0$
 and $\widetilde{C}(p,q)>0$ such that
\begin{align*}
u^{p-q+3}\geq C(p,q)
\quad \mbox{and}\quad 
u^{p-q+2}\geq \widetilde{C}(p,q)
\end{align*}
for all $ r \in (0, R)$ and all $t \in (0, \tmax)$.
In order to show that the conclusion of this lemma holds,
 we pick any $R_0\in (0,R)$ satisfying   
\begin{align*}
R_0\leq\frac{nC(p,q)}{4 c^3_1c_2q\chi\mu}.
\end{align*}
Here, let $m$ be an arbitrary even integer
and introduce 
\begin{align*}
I(t):=p\int_{0}^{R_0} r^{n-1}\frac{u^{m+4}_r}{u^3}\,dr. 
\end{align*}
By using the lower estimate \eqref{u.ulb}  for $u$ we first obtain
 \begin{align}\label{It.lest}
I(t)\geq \frac{p}{c^{3}_3}\int_{0}^{R_0} r^{n-1}u^{m+4}_r\,dr
\geq \frac{1}{c^{3}_3}\int_{0}^{R_0}
 r^{n-1}u^{m+4}_r\,dr  
\end{align}
for all $t\in (0,\tmax)$. On the other hand, 
we multiply the quantity $r^{n-1}u^m_r$ on the
both sides of \eqref{pur4u3}, integrate over $(0,R_0)$
and use \eqref{rapu} to establish that 
\begin{align}\label{It}
I(t)&=p\int_{0}^{R_0} r^{n-1}\frac{u^{m+4}_r}{u^3}\,dr \\
   \nonumber
    &=\int_{0}^{R_0}
     r^{n-1}\frac{{\sqrt{u^2+u^2_r}}^3}{u^{p+1}}u^m_rz\,dr
      -\int_{0}^{R_0} (r^{n-1}u_r)_r\cdot u^m_r\,dr
       \\
\nonumber &\quad\,
      -\int_{0}^{R_0} r^{n-2}\cdot\frac{u^{m+3}_r}{u^2}\,dr 
      -(p-1)\int_{0}^{R_0}
       r^{n-1}\cdot\frac{u^{m+2}_r}{u}\,dr
       \\
    \nonumber &\quad\,
      +\chi\int_{0}^{R_0}
       r^{n-1}\frac{(\mu-u){\sqrt{u^2+u^2_r}}^3u^m_r}
                                {u^{p-q+2}{\sqrt{1+v^2_r}}^3}\,dr
      \\
    \nonumber &\quad\,
      +q\chi\int_{0}^{R_0}
       r^{n-1}\frac{{\sqrt{u^2+u^2_r}}^3\cdot u^{m+1}_rv_r}
          {u^{p-q+3}\sqrt{1+v^2_r}}\,dr
  \\
    \nonumber &\quad\,
      +(n-1)\chi\int_{0}^{R_0}
       r^{n-2}\frac{{\sqrt{u^2+u^2_r}}^3u^m_rv^3_r}
          {u^{p-q+2}{\sqrt{1+v^2_r}}^3}\,dr.                              
 \end{align}
 Since the fact that $m+2$  is even means that  
 the fourth term of the right-hand side
  on \eqref{It} is nonpositive, combining \eqref{It.lest} with 
  \eqref{It} implies that
\begin{align}\label{It.est}
\frac{1}{c^{3}_3}\int_{0}^{R_0}
  r^{n-1}u^{m+4}_r\,dr 
  &\leq
\int_{0}^{R_0}
     r^{n-1}\frac{{\sqrt{u^2+u^2_r}}^3}{u^{p+1}}u^m_rz\,dr
      -\int_{0}^{R_0} (r^{n-1}u_r)_r\cdot u^m_r\,dr
\\
    \nonumber &\quad\,
      -\int_{0}^{R_0}
       r^{n-2}\cdot\frac{u^{m+3}_r}{u^2}\,dr 
\\
    \nonumber &\quad\,
      +\chi\int_{0}^{R_0}
       r^{n-1}\frac{(\mu-u){\sqrt{u^2+u^2_r}}^3u^m_r}
       {u^{p-q+2}{\sqrt{1+v^2_r}}^3}\,dr
 \\
    \nonumber &\quad\,
      +q\chi\int_{0}^{R_0}
       r^{n-1}\frac{{\sqrt{u^2+u^2_r}}^3\cdot u^{m+1}_rv_r}
       {u^{p-q+3}\sqrt{1+v^2_r}}\,dr
 \\
    \nonumber &\quad\,
      +(n-1)\chi\int_{0}^{R_0}
       r^{n-2}\frac{{\sqrt{u^2+u^2_r}}^3u^m_rv^3_r}
       {u^{p-q+2}{\sqrt{1+v^2_r}}^3}\,dr\\
                                     \nonumber
   &=:J_1(t)+J_2(t)+J_3(t)+J_4(t)+J_5(t)+J_6(t)
\end{align}
for all $t\in (0,\tmax)$. 
Then we shall show estimates for $J_i$ ($i\in\{1,2,3,4,5,6\}$) 
 from an argument similar to that in the proof of 
 \cite[Lemma 5.1]{B-W}. 
 Employing  
 the Young inequality and 
 the H$\ddot{\mbox{o}}$lder inequality,
 we have that 
for all $t\in (0,\tmax)$,
\begin{align}\label{J1.est}
J_1(t)\leq
 c_4\left[1+\left(\int_{0}^{R_0}
 r^{n-1}u^{m+4}_r\,dr\right)^{\frac{m+3}{m+4}}\right]
 \left(\int_{0}^{R_0} r^{n-1}z^{m+4}_+\,dr\right)^{\frac{1}{m+4}}
\end{align}      
 with 
 $c_4:=\max\{\frac{c_2}{c_3^{p+1}}\cdot\frac{ R^n_1}{n},
 \frac{2c_2}{c_3^{p+1}}\}\,$ and $R_1:=\max\{1,R\}$, and 
 that 
\begin{align}\label{J23.est}
J_2(t)\leq R^{n-1}\cdot L^{m+1},\quad
J_3(t)\leq \frac{R^{n-1}}{C^2_3}\cdot L^{m+3}
\\ \label{J4.est}
J_4(t)\leq
      c_5\left[1+\left(\int_{0}^{R_0}
       r^{n-1}u^{m+4}_r\,dr\right)^{\frac{m+3}{m+4}}\right]
\end{align}
with $c_5
:=\max\{\frac{c_2\chi\mu}{n\widetilde{C}(p,q)}\cdot\frac{ R^n_1}{n},\,
\frac{2c_2\chi\mu}{n\widetilde{C}(p,q)}\cdot\frac{ R^n_1}{n}\}$, as well as that 
\begin{align}\label{J5.est}
J_5(t) 
\leq\left(c_6+\frac{1}{2c^{3}_3}\int_{0}^{R_0}
 r^{n-1}u^{m+4}_r\,dr\right)
 +\left(c_7+\frac{c_8}{m+4}\int_{0}^{R_0}
  r^{n-1}u^{m+4}_r\,dr\right)
            \end{align}
 with      
 $c_6:=\frac{q\chi\mu c_2R_0^{n+1}}{n(n+1)C(p,q)}$, 
 $c_7:=\frac{2q\chi\mu c_1c_2R_0^{n+1}}{n(n+1)C(p,q)}$  
 and $c_8:=\frac{3q\chi\mu c_1c_2R}{nC(p,q)}$ and that 
\begin{align}\label{J6.est}
J_6(t)\leq
      c_9\left[1+\left(\int_{0}^{R_0}
       r^{n-1}u^{m+4}_r\,dr\right)^{\frac{m+3}{m+4}}\right]
\end{align}
with $c_9
:=\max\{\frac{2c_2(n-1)\chi\mu R^n_0}{3\sqrt{3}n^2\widetilde{C}(p,q)}, 
 \frac{4c_2(n-1)\chi\mu}{3\sqrt{3}\widetilde{C}(p,q)}
 \cdot\frac{ R^n_1}{n}\}$.
In summary, \eqref{J1.est}--\eqref{J6.est} combined with \eqref{It.est} show that 
\begin{align*}
\frac{1}{c^{3}_3}\int_{0}^{R_0} r^{n-1}u^{m+4}_r\,dr 
  &\leq c_4\left[1+\left(\int_{0}^{R_0} 
   r^{n-1}u^{m+4}_r\,dr\right)^{\frac{m+3}{m+4}}\right]
      \left(\int_{0}^{R_0} 
      r^{n-1}z^{m+4}_+\,dr\right)^{\frac{1}{m+4}}\\
    \nonumber &\quad\,
      +R^{n-1}\cdot L^{m+1}+\frac{R^{n-1}}{C^2_3}\cdot L^{m+3}
      +c_5\left[1+\left(\int_{0}^{R_0} 
      r^{n-1}u^{m+4}_r\,dr\right)^{\frac{m+3}{m+4}}\right]
      \\
    \nonumber &\quad\,
      +c_6+\frac{1}{2c^{3}_3}\int_{0}^{R_0} r^{n-1}u^{m+4}_r\,dr
      +c_7+\frac{c_8}{m+4}\int_{0}^{R_0} r^{n-1}u^{m+4}_r\,dr
       \\
    \nonumber &\quad\,
      +c_9\left[1+\left(\int_{0}^{R_0} 
      r^{n-1}u^{m+4}_r\,dr\right)^{\frac{m+3}{m+4}}\right]
\end{align*}
for all $t\in (0,\tmax)$. 
Here, we put $m_0>0$ satisfying
$\frac{c_8}{m_0+4}\leq \frac{1}{4c^{3}_3}$.
Then the above inequality implies that for all $m\geq m_0$, 
\begin{align}
\frac{1}{4c^{3}_3}\int_{0}^{R_0} r^{n-1}u^{m+4}_r\,dr 
   &\leq c_4\left[1+\left(\int_{0}^{R_0} 
   r^{n-1}u^{m+4}_r\,dr\right)^{\frac{m+3}{m+4}}\right]
      \left(\int_{0}^{R_0} 
      r^{n-1}z^{m+4}_+\,dr\right)^{\frac{1}{m+4}}
      \\
    \nonumber &\quad\,
      +c_{10}\left(\int_{0}^{R_0} 
      r^{n-1}u^{m+4}_r\,dr\right)^{\frac{m+3}{m+4}}+c_{11}L^{m+4}
\end{align}
holds with some $c_{10}, c_{11}>0$.
In order to establish the conclusion of this lemma, we fix $t\geq 0$ 
and first deal with the case that there exists a sequence of even numbers
$m=m_j\geq m_0$, $j\in\mathbb{N}$ satisfying $m_j\to \infty$ as 
$j\to \infty$ and 
$
\left(\int_{0}^{R_0} 
      r^{n-1}u^{m+4}_r\,dr\right)^{\frac{m+3}{m+4}}\leq L^{m+4}
$
for all $m\in (m_j)_{j\in\mathbb{N}}$. Then taking the limit 
$j\to \infty$ implies that  
\begin{align*} 
\|u_r(\cdot,t)\|_{L^\infty(0,R_0)}
=\lim_{j\to \infty}\left(\int_{0}^{R_0} 
      r^{n-1}u_r^{m_j+4}\,dr\right)^{\frac{1}{m_j+4}}
      =\lim_{j\to \infty}L^{\frac{m_j+3}{m_j+4}}=L.
\end{align*}
We next consider the case that there is no such a sequence.
Then we can pick   
$\widetilde{m_0}\geq m_0$ 
such that 
\begin{align*}
\left(\int_{0}^{R_0} 
      r^{n-1}u^{m+4}_r\,dr\right)^{\frac{m+3}{m+4}}> L^{m+4}
\end{align*}
for all even $m\geq \widetilde{m_0}$. 
Plugging this inequality into (5.13) and noting the fact that $L \ge 1$, 
we obtain that 
\begin{align*}
\frac{1}{4c^{3}_3}\left(\int_{0}^{R_0} 
      r^{n-1}u^{m+4}_r\,dr\right)^{\frac{1}{m+4}} 
   \leq 2c_4 \left(\int_{0}^{R_0} 
      r^{n-1}z^{m+4}_+\,dr\right)^{\frac{1}{m+4}}
      +c_{10}
      +c_{11}.
\end{align*}
Taking the limit $m\to \infty$, we can see that 
\begin{align*}
\frac{1}{4c^{3}_3}\|u_r(\cdot,t)\|_{L^\infty(0,R_0)}
 \leq 2c_4 \|z_+(\cdot,t)\|_{L^\infty(0,R_0)}
      +c_{10}
      +c_{11}
\end{align*}
holds for all $t\in(0,\tmax)$. 
\end{proof}
%
%
%
%

Lemma \ref{lem;ur;origin} gives us the estimate for 
$\|u_r(\cdot,t)\|_{L^\infty(0,R_0)}$ with some $R_0$. 
This means that we have boundedness of $u_r$ only on $(0,R_0)$. 
 Next, we obtain an estimate for $\|u_r(\cdot,t)\|_{L^\infty(R_0,R)}$.
\begin{lem}\label{lem;ur;boundary} 
 Assume that $\tmax<\infty$, 
 but that $\sup_{(r,t)\in(0,R)\times(0, \tmax)}u(r,t)<\infty$. 
 Then with $R_0\in(0,R)$ taken from Lemma 5.1, 
 for all $t_0>0$ 
   there exists a constant $C>0$ such that
    \begin{align*}
    {\|u_r(\cdot,t)\|}_{L^{\infty}(R_0,R)}\leq C
   \left(1+{\|z_+\|}_{L^{\infty}((0,R_0)\times(t_0,t))}\right)
    \end{align*}
    for all $t \in (t_0, \tmax)$.
 \end{lem}
\begin{proof}
Thanks to Lemma \ref{lem;ur;origin}, 
we can find a constant $c_1>0$ such that 
\begin{align}\label{urr0}
    u_r(R_0,t)\leq c_1
    \left(1+{\|z_+(\cdot,t)\|}_{L^{\infty}(0,R_0)}\right)
\end{align}
    for all $t \in (0,\tmax)$.
Now we pick $t_0\in(0,\tmax)$. In particular, 
\eqref{urr0} implies that, given any $t_1\in(t_0,\tmax)$, we have
 \begin{align}\label{D_1}
    u_r(R_0,t)\leq D_1(t_0,t_1)
  := c_1\left(1+{\|z_+\|}
      _{L^{\infty}((0,R_0)\times(t_0,t_1))}\right)
    \end{align}
    for all $t \in (t_0,t_1)$.
Next, we use the assumption and recall Lemma \ref{lem;rulingout-extinction} to pick $c_2>0$ and $c_3>0$ such that 
\begin{align*}
c_2\leq u(r,t) \leq c_3\qquad
\end{align*}
    for all $r\in(0, R)$ and all $t\in(0, \tmax)$.
Moreover, Lemma \ref{lem2.5} yields existence of constants 
 $c_4>0$ and $c_5>0$ such that
\begin{align*}
|v_r(r,t)|\leq c_4r\quad \mbox{and} \quad|v_{rr}(r,t)|\leq c_5
 \end{align*}
  for all $ r \in (0, R)$ and all $t \in (0, \tmax)$.
 Therefore, the functions $a_3(r,t)$,  $\widetilde{A_3}(r,t)$ 
 and $\widetilde{A_4}(r,t)$ 
 in \eqref{Atilde} can be estimated according to 
 \begin{align}\label{c6}
 a_3(r,t)\leq c_6&
 :=p(p-1)c^{p-2}_3+q(q-1)\chi c^{q-2}_3,\\ \label{c7}
 \widetilde{A_3}(r,t)\leq c_7
 &:=\frac{(2p-1)(n-1)}{R_0}\cdot c^{p-1}_3+p(3p+1)c^{p-1}_3
 \\
  \nonumber &\quad\ 
   +(q+2)\chi c^q_3+q\chi c^{q-1}_3c^2_4R^2c_5
   +q\chi(n-1)c^{q-1}_3c^3_4R^2
    \end{align}
and
 \begin{align}\label{c8}
 \widetilde{A_4}(r,t)\leq c_8
 &:=\frac{n-1}{R_0}\cdot c^{p-1}_3
    +3\chi\mu c^q_3 c_4 Rc_5+3\chi c^{q-1}_3 c_4Rc_5
 \\
  \nonumber &\quad\ 
   +\chi(n-1) c^q_3 c^3_4R+3\chi (n-1) c^q_3c^2_4R{c_5}
 \end{align}
 for all $ r \in (0, R)$ and all $t \in (0, \tmax)$.
We now take $\widetilde{c_6}>0$ and $\widetilde{c_7}>0$ fulfilling
$\widetilde{c_6}>c_6$, $\widetilde{c_7}>c_7$
and 
 \begin{align}\label{4c6c7c8}
4\widetilde{c_6}c_8-{\widetilde{c_7}}^2<0.
\end{align}
Then there exist $n\in\mathbb{N}$ 
and $j\in\{0,1,2,3,4,5\}$ such that
\begin{align*}
  2\cdot\frac{(6n-5+j)\pi}{6\cdot\frac{2}{3}}
  <t_1<2\cdot\frac{(6n-6+j)\pi}{6\cdot\frac{2}{3}}.
\end{align*}
Therefore we can find $\ep>0$ such that
\begin{align*}
 2\cdot\frac{(6n-5+j)\pi}{6\cdot\frac{2}{3}}
 <t_1-\ep<t_1
 <2\cdot\frac{(6n-6+j)\pi}{6\cdot\frac{2}{3}},
\end{align*}
and then there exists $\alpha_0>0$ such that  
\begin{align}\label{talpha-2}
 2\cdot\frac{(6n-5+j)\pi}{6\cdot(\frac{2}{3}+\alpha)}
 <t_1-\ep<t_1
 <2\cdot\frac{(6n-6+j)\pi}{6\cdot(\frac{2}{3}+\alpha)}
\end{align} 
for all $\alpha\in(0,{\alpha}_0)$.
Since $x\tan{\frac{\pi}{3x}}\to\infty$ 
as $x\searrow\frac{2}{3}$, 
 we can find $\alpha_1\in(0,\alpha_0)$ such that
\begin{align*}
 \max\left\{D_1(t_0,t_1),\ 
 \sup_{r\in(0,R)}{ u_r(r,\tmax-\ep)}\right\}
 \leq\frac{1}{2\widetilde{c_6}}\left(\frac{2}{3}+\alpha_1\right)
       \tan{\frac{\pi}{3\cdot\left(\frac{2}{3}+\alpha_1\right)}}
 -\frac{\widetilde{c_7}}{2\widetilde{c_6}}.
\end{align*}
Aided by \eqref{4c6c7c8} and the fact that 
$4\widetilde{c_6}x-{\widetilde{c_7}}^2\to\infty$ as $x\to\infty$, 
we obtain from the intermediate value theorem that there is a 
constant $\widetilde{c_8}>c_8$ such that
\begin{align}\label{alpha1-2}
 \sqrt{4\widetilde{c_6}\widetilde{c_8}
 -{\widetilde{c_7}}^2}=\frac{2}{3}+\alpha_1.
\end{align}
 Combination of \eqref{talpha-2} and \eqref{alpha1-2} with
  $\alpha=\alpha_1$ implies that 
\begin{align*}
 \frac{(6n-6+j)\pi}{6\cdot\sqrt{\frac{4\widetilde{c_6}\widetilde{c_8}
 -{\widetilde{c_7}}^2}{4}}}
 <t_1-\ep<t_1<
 \frac{(6n-5+j)\pi}{6\cdot\sqrt{\frac{4\widetilde{c_6}\widetilde{c_8}
 -{\widetilde{c_7}}^2}{4}}}.
\end{align*}
We define a comparison function $\overline{\varphi}$ by letting
  \begin{equation*}
  \overline{\varphi}(r,t):=    
     \begin{cases}
        D-\frac{\widetilde{c_7}}{2\widetilde{c_6}},
        &t=0,
\\[5mm] 
      \sqrt{\widetilde{C}}
      \tan\left[\tan^{-1}\frac{D}{\sqrt{\widetilde{C}}}
      +\widetilde{c_6}\sqrt{\widetilde{C}}\left(t-\frac{j\pi}
      {6\widetilde{c_6}\sqrt{\widetilde{C}}}\right)\right]
      -\frac{\widetilde{c_7}}{2\widetilde{c_6}},
      &t\in\left(\frac{(6n-6+j)\pi}
          {6\widetilde{c_6}\sqrt{\widetilde{C}}},
          \frac{(6n-5+j)\pi}{6\widetilde{c_6}
          \sqrt{\widetilde{C}}}\right]
     \end{cases}
 \end{equation*}
 for $r\in[R_0,R]$, $t\in[t_0,t_1]$, $n\in\mathbb{N}$ 
 and $j\in\{0,1,2,3,4,5\}$, where
 \begin{align*}
 \widetilde{C}
 :=\frac{4\widetilde{c_6}\widetilde{c_8}-{\widetilde{c_7}}^2}
 {4{\widetilde{c_6}}^2}\quad\mbox{and}\quad 
 D:=\frac{1}{2\widetilde{c_6}}\left(\frac{2}{3}+\alpha_1\right)
      \tan{\frac{\pi}{3\cdot\left(\frac{2}{3}+\alpha_1\right)}}.
 \end{align*}
Here we can verify that 
\[\overline{\varphi}(r,t)\geq\overline{\varphi}(r,0)>0\]
for all $r\in(R_0,R)$ and all $t\in(0, \tmax)$ since 
$\overline{\varphi}$ is a monotonically increasing function 
with respect to 
 $t\in\left(\frac{(6n-6+j)\pi}
          {6\widetilde{c_6}\sqrt{\widetilde{C}}},
          \frac{(6n-5+j)\pi}{6\widetilde{c_6}
          \sqrt{\widetilde{C}}}\right]$ 
 for all $n\in\mathbb{N}$ and all $j\in\{0,1,2,3,4,5\}$.   
 Moreover, from the facts that 
 $\overline{\varphi}_r= \overline{\varphi}_{rr}\equiv0$ and that
 ${\overline{\varphi}}_t=\widetilde{c_6}{\overline{\varphi}}^2
   +\widetilde{c_7}\overline{\varphi}+\widetilde{c_8}$, 
 we use \eqref{c6}, \eqref{c7} and \eqref{c8} to see that  
 with $\mathcal{Q}$ as in \eqref{qphi} we have
  \begin{align*}
   (\mathcal{Q}\overline{\varphi})(r,t)
   &={\overline{\varphi}}_t-a_3(r,t){\overline{\varphi}}^2
   - \widetilde{A_3}(r,t)\overline{\varphi}
   - \widetilde{A_4}(r,t)\\
   &\geq{\overline{\varphi}}_t-|a_3(r,t)|{\overline{\varphi}}^2
   -|\widetilde{A_3}(r,t)|\overline{\varphi}
   -|\widetilde{A_4}(r,t)|\\
   &\geq{\overline{\varphi}}_t-c_6{\overline{\varphi}}^2
   -c_7\overline{\varphi}-c_8\\
   &=(\widetilde{c_6}-c_6){\overline{\varphi}}^2
   +(\widetilde{c_7}-c_7)\overline{\varphi}+(\widetilde{c_8}-c_8)
  \end{align*}
  for all $ r \in (R_0, R)$ and all $t \in [t_1-\ep, t_1]$.
   Then the relations that 
   $\widetilde{c_6}>c_6,\,\widetilde{c_7}>c_7$ 
   and $\widetilde{c_8}>c_8$ 
   ensure that 
 \[(\mathcal{Q}\overline{\varphi})(r,t)>0 \]
   for all $r\in(R_0,R)$ and all $t\in(t_1-\ep, t_1)$.
   Since 
      \[(\mathcal{Q}u_r)(r,t)=0 \]
   for all $(r,t)\in(R_0,R)\times(t_0,t_1)$,
   and since
   \begin{align*}
   u_r(r,t_1-\ep)
     &<\sup_{r\in(0,R)}u_r(r,t_1-\ep)\\
     &\leq\frac{1}{2\widetilde{c_6}}\left(\frac{2}{3}+\alpha_1\right)
            \tan{\frac{\pi}{3\cdot\left(\frac{2}{3}+\alpha_1\right)}}
        -\frac{\widetilde{c_7}}{2\widetilde{c_6}}
     =\overline{\varphi}(r,0)\leq\overline{\varphi}(r,t_1-\ep)
   \end{align*}
for all $r\in[R_0,R]$ and 
 \begin{align*}
  0=u_r(R,t)\leq\overline{\varphi}(R,t)
\end{align*} 
for all $t \in [t_1-\ep, t_1]$ and moreover 
 \begin{align*}
  u_r(R_0,t)\leq D_1(t_0,t_1)
     &\leq \frac{1}{2\widetilde{c_6}}\left(\frac{2}{3}+\alpha_1\right)
            \tan{\frac{\pi}{3\cdot\left(\frac{2}{3}+\alpha_1\right)}}
       -\frac{\widetilde{c_7}}{2\widetilde{c_6}} =\overline{\varphi}(r,0) \leq\overline{\varphi}(r,t) 
\end{align*}  
for all $ r \in [R_0, R]$ and  all $t \in [t_1-\ep, t_1]$, in particular, 
 $u_r(R_0,t)\leq\overline{\varphi}(R_0,t)$ for all $t\in [t_1-\ep, t_1]$, 
the comparison principle derives that
 $u_r(r,t)\leq \overline{\varphi}(r,t)$
for all $ r \in [R_0, R]$ and  all $t \in [t_1-\ep, t_1]$. 
Therefore by putting
\begin{align*}
  C:=  \max\left\{\overline{\varphi}(t_1),\, 
  \max_{(r,t)\in(R_0, R)\times(t_1-\ep, t_1)}{u_r(r,t)}\right\}
\end{align*}
we have this lemma.
\end{proof}
%
%
%
In summary, 
we obtain the following result which shows that $u_r$ is bounded by $z_+$. 
\begin{corollary}\label{ur;esti;z}
 Assume that $\tmax<\infty$, 
 but that $\sup_{(r,t)\in(0,R)\times(0, \tmax)}u(r,t)<\infty$.
 For all $t_0>0$, 
   there exists a constant $C>0$ such that
    \begin{align*}
    {\|u_r(\cdot,t)\|}_{L^{\infty}(0,R)}\leq C
    \left(1+{\|z_+\|}_{L^{\infty}((0,R)\times(t_0,t))}\right)
    \end{align*}
    for all $t \in (t_0, \tmax)$.
\end{corollary}
\begin{proof}
Combination of Lemmas \ref{lem;ur;origin} and \ref{lem;ur;boundary} directly derives this corollary. 
\end{proof}
\subsection{Nonlocal parabolic inequality for\ $z$}
Since our goal is to see that 
$\|u_r(\cdot,t)\|_{L^\infty(0,R)}\leq C$ holds 
for all $t$ with some $C>0$, 
we desire  boundedness of $z_+$. 
Thus it is necessary to observe properties of $z$. 
We first differentiate $z$ with respect to $t$.     
%
%
\begin{lem}\label{lem5.4}
The function $z=\frac{u_t}{u}$ satisfies
\begin{align}\label{z_t}
z_t=B_1(r,t)z_{rr}+B_{21}(r,t)z_r+\frac{B_{22}(r,t)}{r}z_r
+(p-1)z^2+B_3(r,t)z+B_4(r,t)
\end{align} 
 for all $ r \in (0, R)$ and all $t \in (0, \tmax)$, where
   \begin{align}\label{B}
    B_1(r,t)&:=\frac{u^{p+2}}{{\sqrt{u^2+u^2_r}}^3},
    \\ 
    \nonumber
    B_{21}(r,t)&:=2\frac{u^{p+1}u_r}{\sqrt{u^2+u^2_r}}
          -3\frac{u^{p+2}u_ru{rr}}{{\sqrt{u^2+u^2_r}}^5}
          +4p\frac{u^{p-1}u^3_r}{{\sqrt{u^2+u^2_r}}^3}
          -3p\frac{u^{p-1}u^5_r}{{\sqrt{u^2+u^2_r}}^5}\\
    \nonumber &\quad\,          
          +(p-1)\frac{u^{p+1}u_r}{{\sqrt{u^2+u^2_r}}^5}(2u^2-u^2_r)
          -\chi\frac{u^{q-1}v_r}{\sqrt{1+v^2_r}},
    \\
    \nonumber
    B_{22}(r,t)&:=(n-1)\frac{u^{p+2}}{{\sqrt{u^2+u^2_r}}^3},
    \\
    \nonumber
    B_3(r,t)&:=\chi\frac{u^q}{{\sqrt{1+v^2_r}}^3}
          +(p-q)\chi\frac{u^{q-1}}{{\sqrt{1+v^2_r}}^3}
          \left(\mu-u+\frac{n-1}{r}v^3_r\right)
    \\
    \nonumber &\quad\,
         + (pq-2q-1)\chi\frac{u^{q-1}u_rv_r}{u\sqrt{1+v^2_r}},
   \\
    \nonumber
    B_4(r,t)&:=-3\chi\frac{u^{p+q-1}(\mu-u)u_rv_r}
    {\sqrt{u^2+u^2_r}{\sqrt{1+v^2_r}}^5}
               +3{\chi}^2\frac{u^{2q-1}(\mu-u)v^2_r}
    {(1+v^2_r)^3}
    \\
    \nonumber &\quad\,
         +\chi\frac{u^{p+q-2}u^2_r}
    {\sqrt{u^2+u^2_r}{\sqrt{1+v^2_r}}^3}
             -{\chi}^2\frac{u^{2q-2}u_rv_r}
    {(1+v^2_r)^2}
    \\
    \nonumber &\quad\,
    +3\chi\frac{n-1}{r}\cdot\frac{u^{p+q-1}u_rv^2_r}
    {\sqrt{u^2+u^2_r}{\sqrt{1+v^2_r}}^5}
     -3{\chi}^2\frac{n-1}{r}\cdot\frac{u^{2q-1}v^3_r}
    {(1+v^2_r)^3}      
\end{align}  
 for $ r \in (0, R)$ and $t \in (0, \tmax)$.  
\end{lem}
\begin{proof}
The proof is based on an argument in the proof of 
\cite[Lemma 5.4]{B-W}. 
First we differentiate \eqref{z} with respect to $t$ to see that
 \begin{align}\label{zt}
  z_t&=  \left(\frac{u^{p+1}u_{rr}}{{\sqrt{u^2+u^2_r}}^3}\right)_t
 + p\left(\frac{u^{p-2}u^4_r}{{\sqrt{u^2+u^2_r}}^3}\right)_t
 +\frac{n-1}{r}\left(\frac{u^{p-1}u_r}{\sqrt{u^2+u^2_r}}\right)_t
  \\
    \nonumber &\quad\,
 +(p-1)\left(\frac{u^pu^2_r}{{\sqrt{u^2+u^2_r}}^3}\right)_t
 -q\chi\left(\frac{u^{q-2}u_rv_r}{\sqrt{1+v^2_r}}\right)_t
 -\chi\left(\frac{u^{q-1}(\mu-u)}{{\sqrt{1+v^2_r}}^3}\right)_t
  \\
    \nonumber &\quad\,
 -\chi\frac{n-1}{r} \left(\frac{u^{q-1}v^3_r}
{{\sqrt{1+v^2_r}}^3} \right)_t
 \end{align}
  for all $ r \in (0, R)$ and all $t \in (0, \tmax)$.  
  Now, rewriting $u_t, u_{rt}$ and $u_{rrt}$ as $u_t=uz, u_{rt}=uz_r+u_rz$ and $u_{rrt}=uz_{rr}+2u_rz_r+u_{rr}z$, we obtain 
 \begin{align*}
  \left(\frac{u^{p+1}u_{rr}}{{\sqrt{u^2+u^2_r}}^3}\right)_t
& = \frac{u^{p+1}u_{rrt}}{{\sqrt{u^2+u^2_r}}^3}
 +(p+1)\frac{u^pu_tu_{rr}}{{\sqrt{u^2+u^2_r}}^3}
  -\frac{3}{2}\cdot\frac{u^{p+1}u_{rr}(2uu_t+2u_ru_{rt})}
  {{\sqrt{u^2+u^2_r}}^5}\\
& =\frac{u^{p+2}}{{\sqrt{u^2+u^2_r}}^3}z_{rr}
 +2\frac{u^{p+1}u_r}{{\sqrt{u^2+u^2_r}}^3}z_r
 +\frac{u^{p+1}u_{rr}}{{\sqrt{u^2+u^2_r}}^3}z
 +(p+1)\frac{u^{p+1}u_{rr}}{{\sqrt{u^2+u^2_r}}^3}z
  \\
    \nonumber &\quad\,
 -3\frac{u^{p+3}u_{rr}}{{\sqrt{u^2+u^2_r}}^5}z
 -3\frac{u^{p+2}u_ru_{rr}}{{\sqrt{u^2+u^2_r}}^5}z_r
 -3\frac{u^{p+1}u^2_ru_{rr}}{{\sqrt{u^2+u^2_r}}^5}z.
 \end{align*}
Simplifying the third, fourth, fifth and sixth terms 
on this identity according to 
  \begin{align*}
  \frac{u^{p+1}u_{rr}}{{\sqrt{u^2+u^2_r}}^3}z 
  &+(p+1)\frac{u^{p+1}u_{rr}}{{\sqrt{u^2+u^2_r}}^3}z
  -3\frac{u^{p+3}u_{rr}}{{\sqrt{u^2+u^2_r}}^5}z
  -3\frac{u^{p+1}u^2_ru_{rr}}{{\sqrt{u^2+u^2_r}}^5}z\\
  &= \frac{u^{p+1}u_{rr}}{{\sqrt{u^2+u^2_r}}^5}z
  \left((u^2+u^2_r)+(p+1)(u^2+u^2_r)
  -3u^2-3u^2_r\right) \\
  &=(p-1)\frac{u^{p+1}u_{rr}}{{\sqrt{u^2+u^2_r}}^3}z ,
  \end{align*}
 we obtain
  \begin{align}\label{zt1}
  \left(\frac{u^{p+1}u_{rr}}{{\sqrt{u^2+u^2_r}}^3}\right)_t
  &=\frac{u^{p+2}}{{\sqrt{u^2+u^2_r}}^3}z_{rr}
 +\left(2\frac{u^{p+1}u_r}{{\sqrt{u^2+u^2_r}}^3}
 -3\frac{u^{p+2}u_ru_{rr}}{{\sqrt{u^2+u^2_r}}^5}\right)z_r
 \\
    \nonumber &\quad\,
 +(p-1)\frac{u^{p+1}u_{rr}}{{\sqrt{u^2+u^2_r}}^3}z
  \end{align}
   for all $ r \in (0, R)$ and all $t \in (0, \tmax)$.  Similarly, we have
 \begin{align}\label{zt2}
   \left(\frac{u^{p-2}u^4_r}{{\sqrt{u^2+u^2_r}}^3}\right)_t
  &=\left(4\frac{u^{p-1}u^3_r}{{\sqrt{u^2+u^2_r}}^3}
 -3\frac{u^{p-1}u^5_r}{{\sqrt{u^2+u^2_r}}^5}\right)z_r
 +(p-1)\frac{u^{p-2}u^4_r}{{\sqrt{u^2+u^2_r}}^3}z,
 \\ \label{zt3}
  \left(\frac{u^{p-1}u_r}{\sqrt{u^2+u^2_r}}\right)_t
 &=\frac{u^{p+2}}{{\sqrt{u^2+u^2_r}}^3}z_r
 +(p-1)\frac{u^{p-1}u_r}{\sqrt{u^2+u^2_r}}z,
\\ \label{zt4}
  \left(\frac{u^pu^2_r}{{\sqrt{u^2+u^2_r}}^3}\right)_t
 &=\frac{u^{p+1}u_r}{{\sqrt{u^2+u^2_r}}^5}(2u^2-u^2_r)z_r
 +(p-1)\frac{u^pu^2_r}{{\sqrt{u^2+u^2_r}}^3}z.
 \end{align}
 Next, we calculate the fourth term of \eqref{zt}
  and use the relations 
  $u_t=uz$ and $u_{rt}=uz_r+u_rz$ to see that
 \begin{align*}
\left(\frac{u^{q-2}u_rv_r}{\sqrt{1+v^2_r}}\right)_t
 &=\frac{(q-2)u^{q-3}u_tu_rv_r+u^{q-2}u_{rt}v_r+u^{q-2}u_rv_{rt}}
 {\sqrt{1+v^2_r}}
 -\frac{u^{q-2}u_rv^2_rv_{rt}}{\sqrt{1+v^2_r}^3}\\
 \nonumber
 &=(q-1)\frac{u^{q-2}u_rv_r}{\sqrt{1+v^2_r}}z
     +\frac{u^{q-2}v_r}{\sqrt{1+v^2_r}}z_r
     +\frac{u^{q-2}u_r}{\sqrt{1+v^2_r}^3}v_{rt}.
    \end{align*}
 Thanks to Lemma \ref{lem2.4}, we can moreover rewrite $v_{rt}$
 to obtain
\begin{align} \label{zt5}
\left(\frac{u^{q-2}u_rv_r}{\sqrt{1+v^2_r}}\right)_t
 &=(q-1)\frac{u^{q-2}u_rv_r}{\sqrt{1+v^2_r}}z
     +\frac{u^{q-2}v_r}{\sqrt{1+v^2_r}}z_r
     -\frac{u^{p+q-2}u^2_r}
    {\sqrt{u^2+u^2_r}{\sqrt{1+v^2_r}}^3} 
 \\ 
    \nonumber &\quad\,
             +\chi\frac{u^{2q-2}u_rv_r}
    {(1+v^2_r)^2},  
    \end{align}
    as well as
\begin{align} \label{zt6}
  \left(\frac{u^{q-1}(\mu-u)}{{\sqrt{1+v^2_r}}^3}\right)_t
 &=-\frac{u^q}{{\sqrt{1+v^2_r}}^3}z
         +(q-1)\frac{u^{q-1}(\mu-u)}{{\sqrt{1+v^2_r}}^3}z
 \\
    \nonumber &\quad\,
    +3\frac{u^{p+q-1}(\mu-u)u_rv_r}
    {\sqrt{u^2+u^2_r}{\sqrt{1+v^2_r}}^5}
               -3\chi\frac{u^{2q-1}(\mu-u)v^2_r}
    {(1+v^2_r)^3}
\\ \label{zt7}
  \left(\frac{u^{q-1}v^3_r}
{{\sqrt{1+v^2_r}}^3} \right)_t
 &=(q-1)\frac{u^{q-1}v^3_r}
{{\sqrt{1+v^2_r}}^3}z
+3\chi\frac{u^{2q-1}v^3_r}
    {(1+v^2_r)^3}
 -3\frac{u^{p+q-1}u_rv^2_r}
    {\sqrt{u^2+u^2_r}{\sqrt{1+v^2_r}}^5},   
 \end{align}
 for all $ r \in (0, R)$ and all $t \in (0, \tmax)$.
 In summary, \eqref{zt1}--\eqref{zt7} combined with \eqref{zt} 
 show that
 \begin{align}\label{ztlong}
 z_t&=\frac{u^{p+2}}{{\sqrt{u^2+u^2_r}}^3}z_{rr}
 +\left(2\frac{u^{p+1}u_r}{{\sqrt{u^2+u^2_r}}^3}
 -3\frac{u^{p+2}u_ru_{rr}}{{\sqrt{u^2+u^2_r}}^5}\right)z_r 
 +(p-1)\frac{u^{p+1}u_{rr}}{{\sqrt{u^2+u^2_r}}^3}z
  \\
    \nonumber &\quad\,
    +\left(4p\frac{u^{p-1}u^3_r}{{\sqrt{u^2+u^2_r}}^3}
 -3p\frac{u^{p-1}u^5_r}{{\sqrt{u^2+u^2_r}}^5}\right)z_r
 +p(p-1)\frac{u^{p-2}u^4_r}{{\sqrt{u^2+u^2_r}}^3}z
  \\
    \nonumber &\quad\,
 +\frac{n-1}{r} \cdot \frac{u^{p+2}}{{\sqrt{u^2+u^2_r}}^3}z_r
 +(p-1)\frac{n-1}{r}\cdot\frac{u^{p-1}u_r}{\sqrt{u^2+u^2_r}}z
  \\
    \nonumber &\quad\,
  +(p-1)\frac{u^{p+1}u_r}{{\sqrt{u^2+u^2_r}}^5}(2u^2-u^2_r)z_r
 +(p-1)^2\frac{u^pu^2_r}{{\sqrt{u^2+u^2_r}}^3}z
     \\
    \nonumber &\quad\,
    -\chi\frac{u^{q-1}v_r}{\sqrt{1+v^2_r}}z_r
 -(q-1)\chi\frac{u^{q-2}u_rv_r}{\sqrt{1+v^2_r}}z
    +\chi\frac{u^{p+q-2}u^2_r}
    {\sqrt{u^2+u^2_r}{\sqrt{1+v^2_r}}^3}
    \\
    \nonumber &\quad\,
             -{\chi}^2\frac{u^{2q-2}u_rv_r}
    {(1+v^2_r)^2}
    +\chi\frac{u^q}{{\sqrt{1+v^2_r}}^3}z
         -(q-1)\chi\frac{u^{q-1}(\mu-u)}{{\sqrt{1+v^2_r}}^3}z
         \\
    \nonumber &\quad\,
     -3\chi\frac{u^{p+q-1}(\mu-u)u_rv_r}
    {\sqrt{u^2+u^2_r}{\sqrt{1+v^2_r}}^5}
               +3{\chi}^2\frac{u^{2q-1}(\mu-u)v^2_r}
    {(1+v^2_r)^3}
     \\
    \nonumber &\quad\,
    -(q-1)\chi\cdot\frac{n-1}{r}\frac{u^{q-1}v^3_r}
  {{\sqrt{1+v^2_r}}^3}z
-3{\chi}^2\frac{n-1}{r}\cdot\frac{u^{2q-1}v^3_r}
     {(1+v^2_r)^3}
      \\
    \nonumber &\quad\,
     +3\chi\frac{n-1}{r}\cdot\frac{u^{p+q-1}u_rv^2_r}
    {\sqrt{u^2+u^2_r}{\sqrt{1+v^2_r}}^5}  
\end{align}
  for all $ r \in (0, R)$ and all $t \in (0, \tmax)$. 
Now we simplify the third, fifth, seventh, ninth,
   eleventh, fifteenth and eighteenth terms
   on the right-hand side.  
 Recalling the definition of $z$ (see \eqref{z}), 
 we rearrange with the new quantity $(p-1)z^2$ such that 
 \begin{align}\label{p-1z^2}
 &(p-1)\frac{u^{p+1}u_{rr}}{{\sqrt{u^2+u^2_r}}^3}z
 +p(p-1)\frac{u^{p-2}u^4_r}{{\sqrt{u^2+u^2_r}}^3}z
 +(p-1)\frac{n-1}{r}\cdot\frac{u^{p-1}u_r}{\sqrt{u^2+u^2_r}}z
  \\
    \nonumber &\quad\,
 +(p-1)^2\frac{u^pu^2_r}{{\sqrt{u^2+u^2_r}}^3}z
  -(q-1)\chi\frac{u^{q-2}u_rv_r}{\sqrt{1+v^2_r}}z
  -(q-1)\chi\frac{u^{q-1}(\mu-u)}{{\sqrt{1+v^2_r}}^3}z
   \\
    \nonumber &\quad\,
   -(q-1)\chi\frac{n-1}{r}\cdot\frac{u^{q-1}v^3_r}
  {{\sqrt{1+v^2_r}}^3}z
  \\  \nonumber
 &=(p-1)z^2+(pq-2q+1)\chi\frac{u^{q-2}u_rv_r}{\sqrt{1+v^2_r}}z+(p-q)\chi\frac{u^{q-1}(\mu-u)}{{\sqrt{1+v^2_r}}^3}z
  \\
    \nonumber &\quad\,
    +(p-q)\chi\frac{n-1}{r}\cdot\frac{u^{q-1}v^3_r}
     {{\sqrt{1+v^2_r}}^3}z.
 \end{align}
Thus plugging \eqref{p-1z^2} into \eqref{ztlong} implies that 
 \begin{align*}
 z_t&=\frac{u^{p+2}}{{\sqrt{u^2+u^2_r}}^3}z_{rr}
 \\
    \nonumber &\quad\,
 +\left(2\frac{u^{p+1}u_r}{{\sqrt{u^2+u^2_r}}^3}
 -3\frac{u^{p+2}u_ru_{rr}}{{\sqrt{u^2+u^2_r}}^5}\right)z_r 
    +\left(4p\frac{u^{p-1}u^3_r}{{\sqrt{u^2+u^2_r}}^3}
 -3p\frac{u^{p-1}u^5_r}{{\sqrt{u^2+u^2_r}}^5}\right)z_r
 \\
    \nonumber &\quad\,
 +\frac{n-1}{r} \cdot \frac{u^{p+2}}{{\sqrt{u^2+u^2_r}}^3}z_r
  +(p-1)\frac{u^{p+1}u_r}{{\sqrt{u^2+u^2_r}}^5}(2u^2-u^2_r)z_r
    -\chi\frac{u^{q-1}v_r}{\sqrt{1+v^2_r}}z_r
    \\
    \nonumber &\quad\,
    +(p-1)z^2
          +\chi\frac{u^q}{{\sqrt{1+v^2_r}}^3}z
          +(p-q)\chi\frac{u^{q-1}}{{\sqrt{1+v^2_r}}^3}
          \left(\mu-u+\frac{n-1}{r}v^3_r\right)z
    \\
    \nonumber &\quad\,
         + (pq-2q-1)\chi\frac{u^{q-2}u_rv_r}{\sqrt{1+v^2_r}}z
    +\chi\frac{u^{p+q-2}u^2_r}
    {\sqrt{u^2+u^2_r}{\sqrt{1+v^2_r}}^3}
             -{\chi}^2\frac{u^{2q-2}u_rv_r}
    {(1+v^2_r)^2}    
         \\
    \nonumber &\quad\,
     -3\chi\frac{u^{p+q-1}(\mu-u)u_rv_r}
    {\sqrt{u^2+u^2_r}{\sqrt{1+v^2_r}}^5}
               +3{\chi}^2\frac{u^{2q-1}(\mu-u)v^2_r}
    {(1+v^2_r)^3}
-3{\chi}^2\frac{n-1}{r}\cdot\frac{u^{2q-1}v^3_r}
     {(1+v^2_r)^3}
      \\
    \nonumber &\quad\,
     +3\chi\frac{n-1}{r}\cdot\frac{u^{p+q-1}u_rv^2_r}
    {\sqrt{u^2+u^2_r}{\sqrt{1+v^2_r}}^5}
\end{align*}
holds.
\end{proof}
%
%
%
%
Thanks to Corollary \ref{ur;esti;z}, 
we can estimate the right-hand side of \eqref{z_t}.
\begin{lem}\label{zt.est}
Assume that $\tmax<\infty$, 
  but that $\sup_{(r,t)\in(0,R)\times(0, \tmax)}u(r,t)<\infty.$
  Then there exist a constant $d>0$ and continuous functions 
  $b_1$, $b_{21}$, $b_{22}$ and $b_3$ on $[0,R]\times[0,\tmax)$ 
  with properties such that $b_1$ and $b_{22}$ are nonnegative 
  and $z=\frac{u_t}{u}$ satisfies 
  \begin{align}\label{ztlem.est}
  z_t&\leq b_1(r,t)z_{rr}+b_{21}(r,t)z_r+\frac{b_{22}(r,t)}{r}z_r
   \\
    \nonumber &\quad\,
+(p-1)z^2+b_3(r,t)z
+d\left(1+{\|z_+\|}_{L^{\infty}((0,R)\times(t_0,t))}\right)
  \end{align}
   for all $ r \in (0, R)$ and all $t \in (t_0, \tmax)$ 
   and for each $t_0\in(0,\tmax)$. 
\end{lem}
\begin{proof}
 We let 
 \begin{align}\label{Bseries}
 b_1:=B_1,\quad b_{21}:=B_{21},\quad b_{22}:=B_{22}\quad \mbox{and}\quad b_3:=B_3, 
 \end{align}
 where $B_1$, $B_{21}$, $B_{22}$ and $B_3$ 
 are defined in Lemma \ref{lem5.4}. 
 We note that they are continuous in $[0,R]\times[0,\tmax)$, and that $b_1\geq 0$ and  $b_{22}\geq 0$. 
To attain the conclusion we 
will give an estimate for $B_4$ 
defined in Lemma \ref{lem5.4}. Now we 
 again use the condition for $u$ and Lemma \ref{lem2.5} to find constants $c_1, c_2, c_3>0$ such that 
 \begin{align*}
 u(r,t)\leq c_1,\quad
|v_r(r,t)|\leq c_2r\quad\mbox{and}\quad
|v_{rr}(r,t)|\leq c_3
 \end{align*}
 for all $ r \in (0, R)$ and all $t \in (0, \tmax)$. 
 Then we can 
 estimate the first, second, fifth and sixth terms in $B_4$ 
 (see \eqref{B}) as 
 \begin{align}\label{B41}
 -3\chi\frac{u^{p+q-1}(\mu-u)u_rv_r}
    {\sqrt{u^2+u^2_r}{\sqrt{1+v^2_r}}^5}
    \leq 3\chi c^{p+q-1}_1(\mu+c_1)\cdot c_2R
    \end{align}
    and
    \begin{align}\label{B42}
    3{\chi}^2\frac{u^{2q-1}(\mu-u)v^2_r}{(1+v^2_r)^3}
    \leq 3{\chi}^2 c^{2q-1}_1(\mu+c_1)\cdot c^2_2R^2
    \end{align}
    as well as
       \begin{align}\label{B43}
     3\chi\frac{n-1}{r}\cdot\frac{u^{p+q-1}u_rv^2_r}
    {\sqrt{u^2+u^2_r}{\sqrt{1+v^2_r}}^5}
    \leq 3(n-1)\chi c^{p+q-1}_1\cdot c^2_2R
    \end{align}
    and
     \begin{align}\label{B44}
  -3{\chi}^2\frac{n-1}{r}\cdot\frac{u^{2q-1}v^3_r}{(1+v^2_r)^3}
   \leq 3(n-1){\chi}^2 c^{2q-1}_1\cdot c^3_2R^2
    \end{align}
  for all $ r \in (0, R)$ and all $t \in (0, \tmax)$. 
  In the third and fourth terms in $B_4$ (see \eqref{B}), 
  we have estimates such that  
      \begin{align}\label{B45}
    \chi\frac{u^{p+q-2}u^2_r} {\sqrt{u^2+u^2_r}{\sqrt{1+v^2_r}}^3}
    \leq \chi c^{p+q-2}_1|u_r|
    \qquad 
    \mbox{and}
    \qquad
        -{\chi}^2\frac{u^{2q-1}v^3_r}{(1+v^2_r)^3}
         \leq {\chi}^2c^{2q-1}_1|u_r|
      \end{align}
    for all $ r \in (0, R)$ and all $t \in (0, \tmax)$.  
    From \eqref{B41}--\eqref{B45} we obtain that 
\begin{align}\label{B46}
B_4(r,t)&\leq 
     3\chi c^{p+q-1}_1(\mu+c_1)\cdot c_2R
+3{\chi}^2 c^{2q-1}_1(\mu+c_1)\cdot c^2_2R^2
 \\
    \nonumber &\quad\,
    +3(n-1)\chi c^{p+q-1}_1\cdot c^2_2R
 +3(n-1){\chi}^2 c^{2q-1}_1\cdot c^3_2R^2   
 \\
    \nonumber &\quad\,
   +\left(\chi c^{p+q-2}_1+ q {\chi}^2c^{2q-1}_1\right)|u_r|.   
\end{align}
   Here thanks to Corollary \ref{ur;esti;z}, 
   we can find a constant $c_4>0$ satisfying
      \begin{align*}
      |u_r(r,t)|\leq c_4\left(1+{\|z_+\|}_{L^{\infty}((0,R)\times(t_0,t))}\right)
      \end{align*}
      for all $ r \in (0, R)$ and all $t \in (t_0, \tmax)$, which together 
      with \eqref{B46} implies that 
     \begin{align*} 
 B_4(r,t)&\leq 
     c_5+c_6\left(1+{\|z_+\|}_{L^{\infty}((0,R)\times(t_0,t))}\right)   
\end{align*}
with some $c_5, c_6>0$. Therefore we see from \eqref{z_t} and \eqref{Bseries} that \eqref{ztlem.est} holds with 
$d:=c_5+c_6$.
  
\end{proof}
\subsection{Boundedness of\ $z$\ from above}
%
%
In order to estimate the term $z_+$, we introduce the following 
function.
\begin{lem}\label{lemma5.6}
Let $C_1$, $C_2$, $C_3$ and $C_4$ be positive constants 
and satisfy that $\frac{C^2_3-4C_2C_4}{4C^2_2}>0$. 
Assume $M>\sqrt{\widetilde{C}}$, 
with $\widetilde{C}:=\frac{C^2_3-4C_2C_4}{4C^2_2}$.
Then the function defined as
\begin{align*}
g(t):=\frac{2\sqrt{\widetilde{C}}}
{1-De^{-\frac{2C_2\sqrt{\widetilde{C}}}{C_1}t}}
-\frac{C_3}{2C_3}-\sqrt{\widetilde{C}}
\end{align*}
with
\[ 
D:=\frac{M+\frac{C_3}{2C_2}-\sqrt{\widetilde{C}}}
               {M+\frac{C_3}{2C_2}+\sqrt{\widetilde{C}}}
\] 
satisfies 
\begin{align}\label{g'defeq}
C_1g'+C_2g^2+C_3g+C_4=0 
\end{align}
for all $t\geq 0$, and moreover
\begin{align*}
g(t_1)=0, \quad\mbox{where}\quad 
t_1:=\frac{C_1}{2C_2\sqrt{\widetilde{C}}}\log
\frac{D\left(\frac{C_3}{2C_2}+\sqrt{\widetilde{C}}\right)}
{\frac{C_3}{2C_2}-\sqrt{\widetilde{C}}}.
\end{align*}
\end{lem}
\begin{proof}
Straightforward calculations 
lead to the conclusion of this lemma.
\end{proof} 
%
%
%
Now we show boundedness of $z$ from above. 
In the case that $p,q\geq 1$, 
the inequality for $z_t$  includes 
$(p-1)z^2$ and $(pq-2q-1)\chi\frac{u^{q-1}u_rv_r}{u\sqrt{1+v^2_r}}$ 
which do not exist in case $p=q=1$ (see \eqref{ztlem.est}). 
The function $g$ introduced in Lemma \ref{lemma5.6} 
enables us to control these new terms.
\begin{lem}
Assume that $\tmax<\infty$, 
  but that $\sup_{(r,t)\in(0,R)\times(0, \tmax)}u(r,t)<\infty.$ Then
   there exists a constant $C>0$ such that 
   $z=\frac{u_t}{u}$ 
   satisfies 
   \begin{align*}
   z(r,t)\leq C
   \end{align*}
 for all $r \in (0,R)$ and all $t \in (\tmax-\ep, \tmax)$.
\end{lem}
\begin{proof}
We use our condition for $u$ and recall Lemma 3.1 
to pick $c_1>0$ and $c_2>0$ fulfilling
\begin{align}\label{c2uc1}
c_2\leq u(r,t) \leq c_1\qquad
\end{align}
 for all $r \in (0,R)$ and all $t \in (0, \tmax)$,
and apply Lemma 2.5 to find $c_3>0$ such that
\begin{align}\label{vrc3}
|v_r(r,t)|\leq c_3r \qquad
 \end{align}
 for all $r \in (0,R)$ and all $t \in (0, \tmax)$.
Let $M>\sqrt{\widetilde{C}}$ 
with $\widetilde{C}:=\frac{C^2_3-4C_2C_4}{4C^2_2}$ and 
 let  $b_1$, $b_{21}$, $b_{22}$, $b_3$ and $d$ be in Lemma 5.5. 
 Using the function $g$ which is provided by Lemma 5.6, 
 we introduce  
\begin{align*}
\varphi(r,t):=G(t)z(r,t)-dt
\end{align*}
 for $r \in (0,R)$
    and $t \in (0, \tmax)$,
where 
\[G(t):=g\left(t-(n-1)t_1\right)\] 
for all 
$(n-1)t_1<t\leq nt_1$ and all $n\in \mathbb{N}$.
Then, according to Lemma \ref{zt.est}, we have that 
\begin{align}\label{phi-t}
\varphi_t&=G(t) z_t+G_t(t)z-d
 \\ 
   \nonumber
          &\leq G(t) \left(
                    b_1(r,t)z_{rr}+b_{21}(r,t)z_r+\frac{b_{22}(r,t)}{r}z_r
                    +(p-1)z^2+b_3(r,t)z
                        \right)
  \\
    \nonumber &\quad\,
 +dG(t){\|z_+\|}_{L^{\infty}((0,R)\times(t-\ep,t))}+dG(t)+G_t(t)z-d
 \\
    \nonumber
 &= b_1(r,t){\varphi}_{rr}+b_{21}(r,t){\varphi}_r
    +\frac{b_{22}(r,t)}{r}{\varphi}_r
    +\frac{p-1}{G(t)}\cdot(\varphi+dt)^2
  \\
    \nonumber &\quad\,
 +b_3(r,t)(\varphi+dt)
 +dG(t){\|z_+\|}_{L^{\infty}((0,R)\times(t-\ep,t))}+dG(t)
 \\
    \nonumber &\quad\,
 +\frac{G_t(t)}{G(t)}(\varphi+dt)-d
\end{align}
 for all $ r \in (0, R)$ and all $t \in (\tmax-\ep, \tmax)$, and since 
 \[z_r=\left(\frac{u_t}{u}\right)_r=\frac{u_{rt}}{u}-\frac{u_ru_t}{u^2}
 \]
 in $[0,R]\times[0,\tmax)$, the fact that 
 \[u_r(0,t)=u_r(R,t)=0\]
 for all $t\in(\tmax-\ep, \tmax)$ entails that 
 \begin{align}
 \varphi_r(0,t)=
 \varphi_r(R,t)\quad\mbox{for all}\ t\in (\tmax-\ep, \tmax).
 \end{align}
 Here, in order to attain this lemma, we shall show that 
 \[ \varphi(r,t)\leq \|\varphi_+(\cdot,\tmax-\ep)\|_{L^\infty(0,R)}\] 
 by using a contradiction argument. 
 Now, if for some $ T\in (\tmax-\ep, \tmax)$, the value 
 \begin{align*}
 S:=\sup_{(r,t)\in(0,R)\times(\tmax-\ep, T)}\varphi(r,t)<\infty
 \end{align*}
 is positive and is
  attained at some point $ (r_0,t_0)\in[0,R]\times[\tmax-\ep,T]$ 
 with $t_0>\tmax-\ep$, then necessarily
 \begin{align}\label{phi-assumption}
 \varphi_t(r_0,t_0)\geq 0,\quad  \varphi_r(r_0,t_0)=0 \quad \mbox{and}\quad 
 \varphi_{rr}(r_0,t_0)\leq 0.
 \end{align}
 Here, since the case that 
 $t_0=t_1:=\frac{C_1}{2C_2\sqrt{\widetilde{C}}}\log
\frac{D\left(\frac{C_3}{2C_2}+\sqrt{\widetilde{C}}\right)}
{\frac{C_3}{2C_2}-\sqrt{\widetilde{C}}}$ implies 
\[\varphi(r_0,t_0)<0,\]
it is enough to consider the case that 
 $t_0 \neq t_1$.  Thus using \eqref{phi-t} and 
 \eqref{phi-assumption} entails that  
 \begin{align}\label{phi-r0t0}
 0&\leq \varphi_t(r_0,t_0)
 \\   \nonumber
   &\leq  \frac{p-1}{G(t_0)}\cdot(\varphi(r_0,t_0)+dt_0)^2
  + b_3(r_0,t_0)(\varphi(r_0,t_0)+dt_0)
  \\
    \nonumber &\quad\, 
 +dG(t_0){\|z_+\|}_{L^{\infty}((0,R)\times(t_0-\ep,t_0))}+dG(t_0) 
 +\frac{G_t(t_0)}{G(t_0)}(\varphi(r_0,t_0)+dt_0)-d 
 \\   \nonumber 
 &= \frac{1}{G(t_0)}
 (\varphi(r_0,t_0)+dt_0)G_t(t_0)
 + \frac{1}{G(t_0)}dG^2(t_0)
  \\
    \nonumber &\quad\, 
 + \frac{1}{G(t_0)}d{\|z_+\|}_{L^{\infty}((0,R)\times(t_0-\ep,t_0))}G^2(t_0)
 \\
    \nonumber &\quad\, 
 + \frac{1}{G(t_0)}\left(b_3(r_0,t_0)(\varphi(r_0,t_0)+dt_0)-d\right)G(t_0)
 \\
    \nonumber &\quad\, 
 + \frac{1}{G(t_0)}(p-1)(\varphi(r_0,t_0)+dt_0)^2.
 \end{align}
 When the special case $r=0$ holds, by picking a sequence 
 $(r_j)_{j\in \mathbb{N}}\subset (0,R)$ such that 
\[r_j\searrow 0\quad \mbox{as} \  
  j\to \infty\] 
and
\[\varphi_r(r_j,t_0)\leq 0 \quad  \mbox{for all} \ j\in \mathbb{N},\] 
 according to the proof of \cite[Lemma 5.6]{B-W}, it is enough to 
 deal with \eqref{phi-r0t0}.
 Now we shall estimate the first, third and fourth 
  terms on the right-hand side of \eqref{phi-r0t0}.
Since 
\[G_t=g'<0\] 
holds, there exists a constant $c_5>0$ such that 
 \begin{align}\label{G_testi}
  (\varphi(r_0,t_0)+dt_0)G_t(t_0)\geq c_5G_t(t_0).
 \end{align}
 Next 
 we 
obtain that 
 \begin{align}\label{G^2esti}
 {\|z_+\|}_{L^{\infty}((0,R)\times(t_0-\ep,t_0))}G^2(t_0)
    &= G^2(t_0)\sup_{(r,t)\in(0,R)\times(t_0-\ep,t_0)}
    \left\{\frac{\varphi(r,s)+ds}{G(s)}\right\}
     \\     \nonumber
    &\leq G(t_0)\left\{
    \sup_{(r,t)\in(0,R)\times(t_0-\ep,t_0)}\varphi_+(r,s)+dt_0\right\}
     \\     \nonumber
    &=c_6G(t_0),
 \end{align}
with $c_6:=\left(\varphi(r_0,t_0)+dt_0\right)$. 
Recalling the definition of $b_3$ and using the estimates for 
\eqref{c2uc1} and \eqref{vrc3}, we infer that 
\begin{align}\label{G^1esti}
&\left(b_3(r_0,t_0)(\varphi(r_0,t_0)+dt_0)-d\right)G(t_0)
   \\   \nonumber
 &\leq (c_6-d)\left\{\chi\frac{u^q}{{\sqrt{1+v^2_r}}^3}
          +(p-q)\chi\frac{u^{q-1}}{{\sqrt{1+v^2_r}}^3}
          \left(\mu-u+\frac{n-1}{r}v^3_r\right)\right\}G(t_0)
    \\
    \nonumber &\quad\,
     + (c_6-d)(pq-2q-1)\chi\frac{u^{q-1}u_rv_r}{u\sqrt{1+v^2_r}}G(t_0)
      \\   \nonumber
 &\leq (c_6-d)\chi\left\{c_1
          +(pc_1^{q-1}+qc_1^q)+(n-1)|p-q|c_3c_1^{q-1}\right\}G(t_0)
    \\
    \nonumber &\quad\,
     + (c_6-d)|pq-2q-1|\cdot\chi\cdot \frac{c_1^{q-2}}{c_2}|u_r|G(t_0).
 \end{align} 
 Thanks to Corollary 5.3 and \eqref{G^2esti}, we moreover estimate 
  the second term on the right-hand side of \eqref{G^1esti}
to see that 
\begin{align}\label{urG}
|u_r|G(t_0)
&\leq C\cdot{\|z_+\|}_{L^{\infty}((0,R)\times(t_0-\ep,t_0))}G(t_0)
\leq C\cdot c_6.
\end{align}
Then we combine \eqref{G^1esti} and \eqref{urG} to obtain 
\begin{align}\label{G^1estilast}
\left(b_3(r_0,t_0)(\varphi(r_0,t_0)+dt_0)-d\right)G(t_0)
\leq c_7G(t_0)+c_8,
\end{align}
with 
\begin{align*}
c_7:=(c_6-d)\chi\left\{c_1
          +(pc_1^{q-1}+qc_1^q)+(n-1)|p-q|c_3c_1^{q-1}\right\}
\end{align*}
 and
\begin{align*}
     c_8:=C c_6(c_6-d)|pq-2q-1|\cdot\chi\cdot \frac{c_1^{q-2}}{c_2 }.
\end{align*} 
 Thus plugging \eqref{G_testi}, \eqref{G^2esti} and  \eqref{G^1estilast} into \eqref{phi-r0t0} together with the definition 
 of $G$ and \eqref{g'defeq} yields 
 \begin{align*}
 0&\leq \varphi_t(r_0,t_0)
 \\   \nonumber
   &< \frac{1}{G(t_0)}\left\{c_5G_t(t_0)+dG^2(t_0)+dc_6G(t_0)
   +c_7G(t_0)+c_8+(p-1)c_6^2
   \right\}
   \\    \nonumber
   &= \frac{1}{G(t_0)}\left(C_1G_t(t_0)+C_2G^2(t_0)+C_3G(t_0)+C_4
   \right)
   \\    \nonumber
   &=0,
 \end{align*}
 with $C_1:=c_5, C_2:=d, C_3:=dc_6+c_7$ and $C_4:=c_8+(p-1)c_6^2$,
 which contradicts.
 Thus this implies that 
 \begin{align*}
 \varphi(r,t)&\leq \|\varphi_+(\cdot,0)\|_{L^\infty(0,R)}
=\|G(0)z_+(\cdot,\tmax-\ep)\|_{L^\infty(0,R)}
 \end{align*}
 for all $r\in (0,R)$ and all $t\in (\tmax-\ep, \tmax)$.
 Therefore, we establish 
 \begin{align*}
 z(r,t)&\leq \frac{G(0)\|z_+(\cdot,\tmax-\ep)\|_{L^\infty(0,R)}+dt}{G(t)}
\\
 &\leq \frac{G(0)\|z_+(\cdot,0)\|_{L^\infty(0,R)}+d\tmax}{G(\tmax)}
 \end{align*}
  for all $r\in (0,R)$ and $t\in (\tmax-\ep, \tmax)$.
  This completes the proof.
\end{proof}
\subsection{Boundedness of\ $u$\ implies extensibility.\,Proof of Theorem\,1.1.}
We have already established two important estimates from 
Corollary 5.3 and Lemma 5.7 such that   
 \begin{align*}
    {\|u_r(\cdot,t)\|}_{L^{\infty}(0,R)}\leq C
    \left(1+{\|z_+\|}_{L^{\infty}((0,R)\times(t_0,t))}\right)
    \end{align*}
 for all $t \in (t_0, \tmax)$, and 
       \begin{align*}
  0\leq z_+(r,t)\leq C
   \end{align*}
 for all $r \in (0,R)$ and all $t \in (0, \tmax)$.
By combining these estimates we can obtain the desired 
 boundedness of $u_r$. 
 Therefore, we only provide the statement of the corollary. 
%
%
\begin{corollary}\label{cor;rulingourfgradientbu}
Assume that $\tmax<\infty$, 
but that $\sup_{(r,t)\in(0,R)\times(0, \tmax)}u(r,t)<\infty$. Then
   there exists a constant $C>0$ such that 
   \begin{align*}
   {\|u_r(\cdot,t)\|}_{L^{\infty}(0,R)}\leq C
   \end{align*}
 for all $t \in (0, \tmax)$.
\end{corollary}

Now we prove Theorem 1.1. 
%
%
\begin{proof}[{\bf Proof of Theorem 1.1.}]
Thanks to Lemma 2.1, we have already 
known local existence of solutions and 
extensibility criterion 
including extinction and gradient brow-up of solutions. 
Moreover, Lemmas \ref{lem;rulingout-extinction} and 
\ref{cor;rulingourfgradientbu} entail ruling out 
the possibility of extinction and gradient brow-up, 
which implies that 
Theorem \ref{mainthm1} holds.  
\end{proof}
\section{ Boundedness. Proof of Theorem \ref{mainthm2}.}\label{Sec6}
In light of extensibility criterion \eqref{extensibility;criterion}, 
we moreover establish the results not only about global existence 
but also about boundedness of solutions. 
In this section we will prove Theorem 1.2 
through a series of lemmas. 
We first recall the estimate for the term which comes from the 
diffusion term (see \cite[Lemma 6.1]{B-W}).
%
%
\begin{lem}\label{lem6.1}
Let $r\geq1$. Then 
\begin{align*}
\int_{\Omega} u^{r-1}|\nabla u|
\leq\int_{\Omega}\frac{u^{r-1}|\nabla u|^2}{\sqrt{u^2+|\nabla u|^2}} 
+\int_{\Omega} u^r
\end{align*}
for all $t\in(0,\tmax)$.
\end{lem}
We next have the following important inequality which means that 
the quantity $\int_\Omega u^{m+p+\alpha}$ with 
$\alpha < -1+\frac 1n$ is controlled by 
$\int_\Omega |\nabla u^{m+p-1}|$. 
%
%
\begin{lem}\label{lemma6.2}
Let $m\geq1$, $n\in\mathbb{N}$ and 
$\alpha\in(-m-p,-1+\frac{1}{n})$. 
Then  there exists a constant $C>0$ such that 
\begin{align*}
\int_{\Omega} u^{m+p+\alpha}
\leq\eta\int_{\Omega}|\nabla u^{m+p-1}| 
+ C^m \left(\eta^{\frac{-n(m+p+\alpha-1)}{-n\alpha-n+1}}+1 \right)
\end{align*}
holds for all $\eta>0$. 
\end{lem}
\begin{proof}
We first note that 
\begin{align}\label{6.2-1}
\int_{\Omega} u^{m+p+\alpha}&=
\int_{\Omega} u^{(m+p-1)\frac{m+p+\alpha}{m+p-1}}\\ \nonumber
&= \|u^{m+p-1}\|_{L^\theta(\Omega)}^{\theta} 
\end{align}
holds with \[\theta:=\frac{m+p+\alpha}{m+p-1}.\] 
Then using the Gagliardo--Nirenberg type 
inequality (see \cite{L-L}) to obtain 
\begin{align*}
\|u^{m+p-1}\|_{L^\theta(\Omega)}
\leq c_1\|\nabla u^{m+p-1}\|^a_{L^1(\Omega)}\cdot
\|u^{m+p-1}\|^{1-a}_{L^\frac{1}{m+p-1}(\Omega)}
+ c_1\|u^{m+p-1}\|_{L^\frac{1}{m+p-1}(\Omega)} 
\end{align*}  
with \[a:=\frac{n(m+p+\alpha-1)(m+p-1)}
{\{n(m+p-1)-(n-1)\}(m+p+\alpha)}\]
and some $c_1>0$, 
by virtue of the elementary inequality 
$(X+Y)^{\theta} \leq 2^{\theta}(X^{\theta}+Y^{\theta})$
 for $X,Y\geq 0$, we infer from \eqref{6.2-1} that 
\begin{align}\label{6.2-2}
\int_{\Omega} u^{m+p+\alpha}&\leq
 (2 c_1)^\theta\left(\|\nabla u^{m+p-1}\|^{a\theta}_{L^1(\Omega)}
\|u^{m+p-1}\|^{(1-a)\theta}_{L^\frac{1}{m+p-1}(\Omega)}
+\|u^{m+p-1}\|^\theta
_{L^\frac{1}{m+p-1}(\Omega)}\right).
\end{align}
Since the condition 
$\alpha\in(-m-p,-1+\frac{1}{n})$ 
implies that  
$a\theta \in (0,1)$, 
the Young inequality entails that 
\begin{align}\label{6.2-Yineq}
&(2c_1)^\theta\|\nabla u^{m+p-1}\|^{a\theta}
_{L^1(\Omega)}\|u^{m+p-1}\|^{(1-a)\theta}_{L^\frac{1}{m+p-1}(\Omega)}
\\
    \nonumber
&\leq a\theta\eta\|\nabla u^{m+p-1}\|_{L^1(\Omega)}
+(1-a\theta)(2c_1)^{\frac{\theta}{1-a\theta}}
\eta^{\frac{-a\theta}{1-a\theta}}\|u^{m+p-1}\|
^{\frac{(1-a)\theta}{1-a\theta}}_{L^\frac{1}{m+p-1}(\Omega)}
\\
    \nonumber
&\leq \eta\|\nabla u^{m+p-1}\|_{L^1(\Omega)}
+(2c_1)^{\frac{\theta}{1-a\theta}}\eta^{\frac{-a\theta}{1-a\theta}}\|u^{m+p-1}\|
^{\frac{(1-a)\theta}{1-a\theta}}_{L^\frac{1}{m+p-1}(\Omega)}.
\end{align}
Thus plugging \eqref{6.2-Yineq} into \eqref{6.2-2} together with the mass conservation low $\int_\Omega u = \int_\Omega u_0$ yields that 
\begin{align}\label{ineq;gagliYo}
\int_{\Omega} u^{m+p+\alpha}
&\leq \eta\|\nabla u^{m+p-1}\|_{L^1(\Omega)}
+(2c_1)^{\frac{\theta}{1-a\theta}}\eta^{\frac{-a\theta}{1-a\theta}}\|u^{m+p-1}\|
^{\frac{(1-a)\theta}{1-a\theta}}_{L^\frac{1}{m+p-1}(\Omega)}
\\
    \nonumber &\quad\,
+(2c_1)^\theta\|u^{m+p-1}\|^\theta
_{L^\frac{1}{m+p-1}(\Omega)}
\\
    \nonumber
&= \eta\int_{\Omega}|\nabla u^{m+p-1}|
\\
    \nonumber &\quad\,
+(2c_1)^{\theta\cdot \frac{n(m+p-2)+1}{-n\alpha-n+1}}\left(\int_{\Omega} u_0\right)
^{\frac{(n-1)\theta-n}{(n-1)+n\alpha}(m+p-1)}
\eta^{\frac{-n(m+p+\alpha-1)}{-n\alpha-n+1}}
\\
    \nonumber &\quad\,
+(2c_1)^\theta\left(\int_{\Omega} u_0 \right)
^{m+p+\alpha}.
\end{align}
Moreover, the facts that 
$\theta< \frac{m+p-1+\frac 1n}{m+p-1} \le 1+\frac 1n$  
and that $-n\alpha -n +1 >
 0$ 
enable us to find some constant $c_2>0$ such that 
\begin{align}\label{ineq;indepm;coef}
(2c_1)^{\theta\cdot \frac{n(m+p-2)+1}{-n\alpha-n+1}}\left(\int_{\Omega} u_0\right)
^{\frac{(n-1)\theta-n}{(n-1)+n\alpha}(m+p-1)}  
\le c_2^m 
\quad \mbox{and} \quad 
(2c_1)^\theta\left(\int_{\Omega} u_0 \right)
^{m+p+\alpha}\le c_2^m. 
\end{align}
Therefore, a combination of \eqref{ineq;gagliYo} with \eqref{ineq;indepm;coef} 
derives this lemma. 
\end{proof}
Thanks to Lemma 6.2, we can attain the following key inequality 
which is useful not only for obtaining 
a differential inequality for
 $\int_\Omega u^m$ for $m\ge 1$ 
but also for showing an $L^\infty$-estimate for $u$ 
via using the Moser iteration argument. 
%
%
\begin{lem}\label{lemma6.3}
Assume that \eqref{pq-Assumption}. 
Then there exist $C_1,C_2,C_3,C_4>0$ such that for all $m\ge 1$,  
\begin{align}\label{lem6.3}
 \frac{d}{dt}\int_{\Omega}u^m&+\int_{\Omega}u^{m}
 +\frac{m(m-1)}{2} \int_{\Omega} u^{m+p-2}|\nabla u|
 \\   
  \nonumber
  &\leq m(m-1)\int_{\Omega} u^{m+p-1}
+C^m_1+C_2 m+C_3m\cdot  C^m_4
\end{align} 
holds on $(0,\tmax)$. 
\end{lem}
 \begin{proof}
 Let $m\geq 1$.
 By multiplying $mu^{m-1}$ on 
the both sides of the first equation in \eqref{P}
 we obtain 
\begin{align}\label{6int1}
\frac{d}{dt}\int_{\Omega}u^m
+m(m-1)\int_{\Omega}
              \frac{u^{m+p-2}|\nabla u|^2}{\sqrt{u^2+|\nabla u|^2}} 
 =m(m-1)\chi\int_{\Omega}
              \frac{u^{m+q-2}\nabla u\cdot\nabla v}{\sqrt{1+|\nabla v|^2}} 
\end{align} 
for all $t\in(0,\tmax)$. Using the second equation in \eqref{P},  
we rewrite the right-hand side of \eqref{6int1} to obtain 
\begin{align}\label{6int2}
m(m&-1)\chi\int_{\Omega}
              \frac{u^{m+q-2}\nabla u\cdot\nabla v}{\sqrt{1+|\nabla v|^2}} 
              \\
              \nonumber
    &=-\frac{m(m-1)\chi}{m+q-1}\int_{\Omega}u^{m+q-1}\nabla\cdot
              \left(\frac{\nabla v}{\sqrt{1+|\nabla v|^2}}\right)
               \\
              \nonumber
     &=-\frac{m(m-1)\chi}{m+q-1}\int_{\Omega}u^{m+q-1}
              \left(\Delta v\frac{1}{\sqrt{1+|\nabla v|^2}}
              +\nabla v\cdot\nabla 
              \left(\frac{1}{\sqrt{1+|\nabla v|^2}}\right)\right)
               \\
              \nonumber
     &=\frac{m(m-1)\chi}{m+q-1}\int_{\Omega}u^{m+q-1}
              \frac{u-\mu}{\sqrt{1+|\nabla v|^2}}
               \\
    \nonumber &\quad\,
         +\frac{m(m-1)\chi}{2(m+q-1)}\int_{\Omega}
         \frac{u^{m+q-1}}{{\sqrt{1+|\nabla v|^2}}^3}
         \nabla v\cdot\nabla(|\nabla v|^2)            
\end{align}
for all $t\in(0,\tmax)$. 
Since 
\begin{equation}\label{nabla}
\nabla v\cdot\nabla(|\nabla v|^2)
=2^n\sum_{i,j}^n \left(\frac{\partial v}{\partial x_i}\cdot
  \frac{\partial v}{\partial x_j}\right)
  \frac{{\partial}^2 v}{\partial x_i\partial x_j}
\end{equation}
holds, we combine \eqref{6int2} with \eqref{nabla} to obtain
\begin{align}\label{6int3}
m(m&-1)\chi\int_{\Omega}
              \frac{u^{m+q-2}\nabla u\cdot\nabla v}{\sqrt{1+|\nabla v|^2}} 
              \\
              \nonumber
     &=\frac{m(m-1)\chi}{m+q-1}\int_{\Omega}u^{m+q-1}
              \frac{u-\mu}{\sqrt{1+|\nabla v|^2}}
               \\
    \nonumber &\quad\,
         +\frac{2^{n-1}m(m-1)\chi}{m+q-1}\sum_{i,j}^n
         \int_{\Omega} \frac{u^{m+q-1}}{{\sqrt{1+|\nabla v|^2}}^3}
         \left(\frac{\partial v}{\partial x_i}\cdot
         \frac{\partial v}{\partial x_j}\right)
         \frac{{\partial}^2 v}{\partial x_i\partial x_j}           
\end{align}
for all $t\in(0,\tmax)$. According to  Lemma \ref{lem2.5},
 moreover we can rearrange 
\begin{align}\label{6int4}
\int_{\Omega}
   & \frac{u^{m+q-1}}{{\sqrt{1+|\nabla v|^2}}^3}
        \left(\frac{\partial v}{\partial x_i}\cdot
  \frac{\partial v}{\partial x_j}\right)
  \frac{{\partial}^2 v}{\partial x_i\partial x_j}  
     \\ 
       \nonumber
  &=\omega_n  \int_0^R
       \frac{u^{m+q-1}v^2_r}{{\sqrt{1+v^2_r}}^3}\cdot v_{rr}  
        \cdot r^{n-1}\,dr
     \\ 
       \nonumber
   &=\omega_n \int_0^R
       \frac{u^{m+q-1}v^2_r}{{\sqrt{1+v^2_r}}^3}
        \left( \frac{{\mu}}{n}-u
                   +\frac{n-1}{r^n}\cdot\int_{0}^{r}
                      \rho^{n-1}u(\rho,t)\,d\rho  \right)
        r^{n-1}\,dr 
     \\ 
       \nonumber
   &= \frac{{\mu\omega_n}}{n} \int_0^R
       \frac{u^{m+q-1}v^2_r}{{\sqrt{1+v^2_r}}^3}\cdot r^{n-1}\,dr
       -\omega_n \int_0^R
       \frac{u^{m+q}v^2_r}{{\sqrt{1+v^2_r}}^3}\cdot r^{n-1}\,dr
        \\
    \nonumber &\quad\,            
        +\omega_n(n-1) \int_0^R
        \frac{u^{m+q-1}v^2_r}{{\sqrt{1+v^2_r}}^3}           
        \cdot\frac{1}{r}\left(\int_{0}^{r} \rho^{n-1}u(\rho,t)\,d\rho
        \right)\,dr 
\end{align}
for all $t\in(0,\tmax)$. Then \eqref{6int2}, \eqref{6int3}, 
and \eqref{6int4}, combined with \eqref{6int1} show that
\begin{align}\label{6int5}
\frac{d}{dt}&\int_{\Omega}u^m
+m(m-1)\int_{\Omega}
              \frac{u^{m+p-2}|\nabla u|^2}{\sqrt{u^2+|\nabla u|^2}} 
              \\ 
       \nonumber
&=\frac{m(m-1)\chi}{m+q-1}\int_{\Omega}
              \frac{u^{m+q}}{\sqrt{1+|\nabla v|^2}}
     -\frac{m(m-1)\chi\mu}{m+q-1}\int_{\Omega}
              \frac{u^{m+q-1}}{\sqrt{1+|\nabla v|^2}}
               \\
    \nonumber &\quad\,
    +\frac{n^2 2^{n-1}m(m-1)\chi\mu\omega_n}{(m+q-1)n}
        \int_0^R
       \frac{u^{m+q-1}v^2_r}{{\sqrt{1+v^2_r}}^3}\cdot r^{n-1}\,dr
                \\
    \nonumber &\quad\,
       -\frac{n^2 2^{n-1}m(m-1)\chi\omega_n}{(m+q-1)n}
        \int_0^R
       \frac{u^{m+q}v^2_r}{{\sqrt{1+v^2_r}}^3}\cdot r^{n-1}\,dr
              \\
    \nonumber &\quad\,
      +\frac{n^2 2^{n-1}m(m-1)\chi\omega_n(n-1)}{(m+q-1)n}
        \int_0^R
      \frac{u^{m+q-1}v^2_r}{{\sqrt{1+v^2_r}}^3}           
        \cdot\frac{1}{r}\left(\int_{0}^{r} \rho^{n-1}u(\rho,t)\,d\rho
        \right)\,dr 
\end{align}
for all $t\in(0,\tmax)$. We apply Lemma \ref{lem6.1} with 
$r=m+p-1$ to establish that 
\begin{align*}
m(m-1)\int_{\Omega}
              \frac{u^{m+p-2}|\nabla u|^2}{\sqrt{u^2+|\nabla u|^2}} 
        \geq m(m-1)\int_{\Omega} u^{m+p-2}|\nabla u|
               -m(m-1)\int_{\Omega} u^{m+p-1}.
\end{align*} 
Then noticing that the second and fourth terms 
on the right-hand side of
 \eqref{6int5} are nonpositive,
 we add $\int_{\Omega}u^{m}$ on the
both sides of \eqref{6int5} to obtain 
 \begin{align}\label{6int6}
 \frac{d}{dt}&\int_{\Omega}u^m+\int_{\Omega}u^{m}
 +m(m-1)\int_{\Omega} u^{m+p-2}|\nabla u|
             \\ 
       \nonumber
       & \le I_1(t)+I_2(t)+I_3(t)+I_4(t)+I_5(t)
\end{align}
for all $t\in(0,\tmax)$, 
where
\begin{align*}
I_1(t)&:= m(m-1)\int_{\Omega} u^{m+p-1}, 
\\
I_2(t) &:= \int_{\Omega}u^{m},
\\
I_3(t) &:= \frac{m(m-1)\chi}{m+q-1}\int_{\Omega}u^{m+q}, 
\\
I_4(t) & := \frac{n2^{n-1}m(m-1)\chi\mu}{(m+q-1)}
        \int_{\Omega}u^{m+q-1}
        \end{align*}
as well as
\begin{align*}
I_5(t) := \frac{n^2 2^{n-1}m(m-1)\chi\omega_n}{m+q-1}
        \int_0^R
      \frac{u^{m+q-1}v^2_r}{{\sqrt{1+v^2_r}}^3}           
        \cdot\frac{1}{r}\left(\int_{0}^{r} \rho^{n-1}u(\rho,t)\,d\rho
        \right)\,dr 
 \end{align*}
 Now, from the condition for 
 $p$ and $q$ (see \eqref{pq-Assumption}) 
 we can take 
 $\ep\in (0,p-q-1+\frac{1}{n})$ 
 and put 
 $\alpha:=-p+q+\ep<-1+\frac{1}{n}$. 
Then we apply the H$\ddot{\mbox{o}}$lder inequality 
and the Young inequality to estimate
\begin{align}\label{6-I2esti}
 I_2(t)&\leq
 \left(\int_{\Omega} u^{m\cdot\frac{m+p+\alpha}{m}}\right)^\frac{m}{m+p+\alpha}
 \left(\int_{\Omega} 1\right)^\frac{p+\alpha}{m+p+\alpha}
 \\ 
       \nonumber
 &=\left(
 \frac{m(m-1)}{m+q-1}\int_{\Omega} u^{m+p+\alpha}
 \right)^\frac{m}{m+p+\alpha}
 \left(\frac{m(m-1)}{m+q-1}\right)
 ^\frac{-m}{m+p+\alpha}|\Omega|^\frac{p+\alpha}{m+p+\alpha}
 \\ 
       \nonumber
 &\leq\frac{m}{m+p+\alpha}\cdot\frac{m(m-1)}{m+q-1}
 \int_{\Omega} u^{m+p+\alpha}
   +\frac{p+\alpha}{m+p+\alpha} 
   \left(\frac{m+q-1}{m(m-1)}\right)^\frac{m}{p+\alpha} |\Omega|
  \\ 
       \nonumber
&\leq\frac{m(m-1)}{m+q-1}\int_{\Omega} u^{m+p+\alpha}
   +\left(\frac{m+q-1}{m(m-1)}\right)^\frac{m}{p+\alpha} |\Omega|,  
 \end{align}
and similarly,
  \begin{align}\label{6-I3esti}
 I_3(t)&\leq\frac{m(m-1)\chi}{m+q-1}\left(\int_{\Omega} u^{m+p+\alpha}+|\Omega|\right)
 \end{align}
 as well as
 \begin{align}\label{6-I4esti}
 I_4(t)&\leq\frac{n2^{n-1}m(m-1)\chi\mu}{(m+q-1)}
                \left(\int_{\Omega} u^{m+p+\alpha}+|\Omega|\right)
 \end{align}
for all $t\in(0,\tmax)$. 
On the other hand, since the H$\ddot{\mbox{o}}$lder inequality
 implies that
\begin{align*}
{\omega}_n\int_0^R
     & \frac{u^{m+q-1}v^2_r}{{\sqrt{1+v^2_r}}^3}           
        \cdot\frac{1}{r}\left(\int_{0}^{r} \rho^{n-1}u(\rho,t)\,d\rho
        \right)\,dr 
        \\
    \nonumber   
     &={\omega}_n\int_0^R u^{m+q-1}\frac{v^2_r}{{\sqrt{1+v^2_r}}^3}           
        \cdot\frac{1}{r}\left(\int_{0}^{r}
         \rho^{(n-1)\left(1-\frac{1}{m+q}\right)}
         \cdot\rho^{\frac{n-1}{m+q}}
        u(\rho,t)\,d\rho
        \right)\,dr 
         \\
    \nonumber   
     &\leq\int_0^R u^{m+q-1}\cdot\frac{1}{r}
     \left(\int_{0}^{r}\rho^{n-1}\right)^{1-\frac{1}{m+q}}\cdot
      {\omega}_n\left(\int_{0}^{r}\rho^{n-1}u^{m+q}(\rho,t)\,d\rho\right)
     ^{\frac{1}{m+q}}\,dr 
     \\
    \nonumber   
     &={\left(\frac{1}{n}\right)}^\frac{m+q-1}{m+q} \|u\|_{L^{m+q}(\Omega)}
     \int_0^R u^{m+q-1}\cdot r^{n\frac{m+q-1}{m+q}-1}
\end{align*}
and also that
\begin{align*}
\int_0^R u^{m+q-1}\cdot r^{n\frac{m+q-1}{m+q}-1}
   &=\int_0^R u^{m+q-1}\cdot r^{(n-1)\frac{m+q-1}{m+q+\ep}}\cdot
         r^{\frac{\ep\left\{n(m+q)-(m+q)-n\right\}-(m+q)}{(m+q)(m+q+\ep)}}
     \\
    \nonumber   
   &\leq\left(\int_0^R u^{m+q+\ep}\cdot r^{n-1}\right)
   ^\frac{m+q-1}{m+q+\ep}
     \left(\int_0^R r^{\frac{\ep n}{1+\ep}\left(1-\frac{1}{m+q}\right)-1}\right)
   ^\frac{1+\ep}{m+q+\ep}
    \\
    \nonumber   
   &=\left(\frac{1+\ep}{\ep n\left(1-\frac{1}{m+q}\right)}\right)
        ^\frac{1+\ep}{m+q+\ep}\cdot
        R^{\frac{\ep n}{m+q+\ep}\left(1-\frac{1}{m+q}\right)}\cdot
        \|u\|^{m+q-1}_{L^{m+q+\ep}(\Omega)},
\end{align*}
it holds that 
\begin{align}\label{6It53}
I_5(t)\leq \frac{m(m-1)}{m+q-1}\cdot C(\ep,n,m,q,R,\chi)
                                      \|u\|^{m+q-1}_{L^{m+q+\ep}(\Omega)}
                              \cdot\|u\|_{L^{m+q}(\Omega)},
\end{align}
where 
\[C(\ep,n,m,q,R,\chi):=n2^{n-1}
              \left(\frac{(1+\ep)(m+q)}{\ep(m+q-1)}\right)
              ^\frac{1+\ep}{m+q+\ep}\cdot
              R^{\frac{\ep n(m+q-1)}{(m+q+\ep)(m+q)}}\chi. \] 
Now, we use the Young inequality  
and the relation
$q+\ep=p+\alpha$
 to see that 
\begin{align}\label{6It54}
\|u\|^{m+q-1}_{L^{m+q+\ep}(\Omega)}\cdot\|u\|_{L^{m+q}(\Omega)}
&\leq \frac{m+q-1}{m+q+\ep}\cdot\|u\|
        ^{m+q-1\cdot\frac{m+q+\ep}{m+q-1}}_{L^{m+q+\ep}(\Omega)}
        +\frac{1+\ep}{m+q+\ep}\cdot\|u\|^{\frac{m+q+\ep}{1+\ep}}
          _{L^{m+q}(\Omega)}
           \\
    \nonumber 
&\leq \int_{\Omega} u^{m+p+\alpha}
        +\|u\|^{\frac{m+q+\ep}{1+\ep}}_{L^{m+q}(\Omega)}.
\end{align}
Since the H$\ddot{\mbox{o}}$lder inequality, the Young inequality
and the relation $q+\ep=p+\alpha$ entail that 
\begin{align}\label{6It55}
 \|u\|^{\frac{m+q+\ep}{1+\ep}}_{L^{m+q}(\Omega)} 
    &\leq \left[ \left(\int_{\Omega} u
    ^{(m+q)\frac{m+p+\alpha}{m+q}}\right)^{\frac{m+q}{m+p+\alpha}}
    \left(\int_{\Omega} 1 \right)^{\frac{p+\alpha-q}{m+p+\alpha}}
 \right]^\frac{m+q+\ep}{(m+q)(1+\ep)}
 \\
    \nonumber 
&=\left(\int_{\Omega} u
    ^{m+p+\alpha}\right)^{\frac{1}{1+\ep}}
   |\Omega|^{\frac{\ep}{(m+q)(1+\ep)}}
    \\
    \nonumber 
&\leq\frac{1}{1+\ep}\int_{\Omega} u^{m+p+\alpha}
   +\frac{\ep}{1+\ep}|\Omega|^{\frac{1}{m+q}},
\end{align}
plugging \eqref{6It54} and \eqref{6It55} into \eqref{6It53} implies
\begin{align}\label{6-I5esti}
I_5(t)\leq
\frac{m(m-1)}{m+q-1}\cdot C(\ep,n,m,q,R,\chi)
 \left(2\int_{\Omega} u^{m+p+\alpha}+|\Omega|^{\frac{1}{m+q}}\right)
\end{align}
for all $t\in(0,\tmax)$. Then by combining \eqref{6-I2esti},
 \eqref{6-I3esti}, \eqref{6-I4esti} and \eqref{6-I5esti} we obtain that 
\begin{align}\label{I_2345}
I_2(t)&+I_3(t)+I_4(t)+I_5(t)
 \\
    \nonumber
&\leq \frac{m(m-1)}{m+q-1}\int_{\Omega} u^{m+p+\alpha}
   +\left(\frac{m+q-1}{m(m-1)}\right)^\frac{m}{p+\alpha} |\Omega|  
\\
    \nonumber &\quad\,       
      +\frac{m(m-1)\chi}{m+q-1}\left(\int_{\Omega} u^{m+p+\alpha}+|\Omega|\right)
      \\
    \nonumber &\quad\, 
      +\frac{n2^{n-1}m(m-1)\chi\mu}{(m+q-1)}
                \left(\int_{\Omega} u^{m+p+\alpha}+|\Omega|\right)
      \\
    \nonumber &\quad\, 
      +\frac{m(m-1)}{m+q-1}\cdot C(\ep,n,m,q,R,\chi)
 \left(2\int_{\Omega} u^{m+p+\alpha}+|\Omega|^{\frac{1}{m+q}}\right)
  \\
    \nonumber
 &\leq \frac{m(m-1)B(m)}{m+q-1}\int_{\Omega} u^{m+p+\alpha}+\widetilde{C}(m),
\end{align}
where 
\begin{align*}
B(m):=1+\chi+n2^{n-1}\chi\mu+2C(\ep,n,m,q,R,\chi)
\end{align*}
and 
\begin{align*}
\widetilde{C}(m)&:=
\left(\frac{m+q-1}{m(m-1)}\right)^\frac{m}{p+\alpha} |\Omega|+
\frac{m(m-1)}{m+q-1}\left(\chi+n2^{n-1}\chi\mu\right)|\Omega| 
  \\
    \nonumber &\quad\, 
+\frac{m(m-1)}{m+q-1}C(\ep,n,m,q,R,\chi)|\Omega|^\frac{1}{m+q}
\end{align*}
    for all $t\in(0,\tmax)$. Therefore, aided by \eqref{I_2345} and 
    Lemma \ref{lemma6.2} with 
    \[\eta:=\frac{m+q-1}{2(m+p-1)B(m)},\] we have 
 \begin{align}\label{I_23456}
& I_2(t)+I_3(t)+I_4(t)+I_5(t)
 \\
    \nonumber 
&\leq \frac{m(m-1)B(m)}{m+q-1}\left\{\eta\int_{\Omega}|\nabla u^{m+p-1}| 
+c^m_1 \left(\eta^{\frac{-n(m+p+\alpha-1)}{-n\alpha-n+1}}+1 \right)\right\}
+\widetilde{C}(m)
\\
    \nonumber
&=\frac{m(m-1)}{2}\int_{\Omega}u^{m+p-2}|\nabla u| 
 \\
    \nonumber &\quad\, 
+\frac{m(m-1)B(m)}{m+q-1}\cdot
c^m_1\left\{\left(\frac{m+q-1}{2(m+p-1)B(m)}\right)^{\frac{-n(m+p+\alpha-1)}{-n\alpha-n+1}}+1\right\}+\widetilde{C}(m)
 \end{align}
with some $c_1>0$. 
Here, to estimate the second, third, fourth terms on
 the right-hand side of \eqref{I_23456}, we first show that 
\begin{align}\label{B(m)}
1\leq B(m)&=1+\chi+n2^{n-1}\chi\mu+2C(\ep,n,m,q,R,\chi)
 \\ 
       \nonumber
 &\leq 1+\chi+n2^{n-1}\chi\mu+n2^{n-1}\cdot
 \frac{2(1+\ep)}{\ep}R^n\chi
 \\ 
       \nonumber
 &\leq c_2
\end{align}
with $c_2:=1+\chi+n2^{n-1}\chi\mu+n2^{n+1}R^n\chi$ and
 \begin{align}\label{C(m)til}
 \widetilde{C}(m)
 &\leq
 \left(2^\frac{1}{p+\alpha}\right)^m|\Omega|+
\left(\chi+n2^{n-1}\chi\mu\right)|\Omega|m
+n2^{n+1}R^n\chi|\Omega|m
 \\[1mm]
       \nonumber
&\leq  c^m_3+c_4 m
 \end{align}
with $c_3:=2^\frac{1}{p+\alpha}|\Omega|$ and 
        $c_4:=\left(\chi+n2^{n-1}\chi\mu
        +n2^{n+1}R^n\chi\right)|\Omega|$.
 Now, plugging \eqref{I_23456} with \eqref{B(m)} and \eqref{C(m)til} 
  into \eqref{6int6}, we can see
 \begin{align*}
 \frac{d}{dt}&\int_{\Omega}u^m+\int_{\Omega}u^{m}
 +\frac{m(m-1)}{2}\int_{\Omega}u^{m+p-2}|\nabla u| 
             \\ 
       \nonumber
       &\leq m(m-1)\int_{\Omega} u^{m+p-1}
+mc^m_1c_2
\left(\frac{m+q-1}{2(m+p-1)B(m)}\right)^{\frac{-n(m+p+\alpha-1)}{-n\alpha-n+1}}
\\[1mm]
    \nonumber &\quad\, 
+mc^m_1c_2+ c^m_3+c_4 m.
  \end{align*} 
  Noting from \eqref{B(m)} that
  \begin{align}\label{Lem6.3;last1esti}  
  \left(\frac{m+q-1}{2(m+p-1)B(m)}\right)^{-n}
  &\leq \left(2\cdot\frac{m+p-1}{m+q-1}\cdot B(m)\right)^n\\[1mm] \nonumber
  &\leq \left(2\cdot\left(1+\frac{p-1}{m}\right)\cdot B(m)\right)^n\\[1mm] \nonumber
  &\leq (2pc_2)^n
  \end{align}
 for all $m\ge 1$, 
  we attain 
  that 
  \begin{align*}
   \frac{d}{dt}&\int_{\Omega}u^m+\int_{\Omega}u^{m}
 +\frac{m(m-1)}{2}\int_{\Omega}u^{m+p-2}|\nabla u| \\[1mm]
 \nonumber
       &\leq m(m-1)\int_{\Omega} u^{m+p-1}
+mc^m_1c_2(2p c_2)^{\frac{n(m+p+\alpha-1)}{-n\alpha-n+1}}
+mc^m_1c_2+ c^m_3+c_4 m
\\[1mm] 
       \nonumber
       &\leq m(m-1)\int_{\Omega} u^{m+p-1}
+C^m_1+C_2 m+C_3m\cdot  C^m_4
 \end{align*}
 with some $C_1, C_2, C_3, C_4>0$, which concludes the proof. 
 \end{proof}

%
%
%
%
 
Combination of Lemmas \ref{lemma6.2} and \ref{lemma6.3} 
implies the following lemma which has an important role in 
obtaining the $L^\infty$-estimate for $u$ in Lemma \ref{lemma6.5}. 

\begin{lem}\label{lem;Lpesti;u}
For all $m\ge 1$ there is $C=C(m)$ such that 
\[
 \lp{m}{u(\cdot,t)} \le C  
\]
for all $t\in (0,\tmax)$ and that $C=C(m)\to \infty$ as $m\to \infty$.
\end{lem}
\begin{proof}
In light of Lemma \ref{lemma6.3} we see that 
 \begin{align}\label{ineq;defiq;withterm}
 \frac{d}{dt}\int_{\Omega}u^m&+\int_{\Omega}u^{m}
 +\frac{m(m-1)}{2} \int_{\Omega} u^{m+p-2}|\nabla u|
 \\   
  \nonumber
  &\leq m(m-1)\int_{\Omega} u^{m+p-1}
 +c^m_1+c_2 m+c_3m\cdot  c^m_4
\end{align} 
holds on $(0,\tmax)$ with some $c_1,c_2,c_3,c_4>0$.  
Here, using Lemma \ref{lemma6.2} with $\eta = \frac{1}{2(m+p-1)}$ to obtain that 
\[
 \int_\Omega u^{m+p-1} \le \frac{m(m-1)}{2(m+p-1)} 
 \io |\nabla u^{m+p-1}| 
 + c_5^m\left( (2(m+p-1))^{n(m+p-2)}+1 \right), 
\]
we infer from \eqref{ineq;defiq;withterm}  
that 
\[
 \frac d{dt} \io u^m + \io u^m 
 \le C(m) 
\]
with 
\[
  C(m):= c^m_1+c_2 m+c_3m\cdot  c^m_4  
 +m(m-1)c_5^m \left( (2(m+p-1))^{n(m+p-2)}+1\right),
\] 
which with the ODE comparison principle means that 
\[
 \lp{m}{u(\cdot,t)} \le \max\left\{ \lp{m}{u_0}, C(m)^{\frac 1m} \right\}. 
\]
Moreover, in view of the fact that 
\begin{align*}
C(m)^{\frac 1m} &\ge [m(m-1)]^{1/m}c_5 m^{n(1+\frac{p-2}{m})}\\
 &\to \infty 
\end{align*}
as $m\to \infty$, 
we can attain this lemma. 
\end{proof}

The estimate obtained in Lemma \ref{lem;Lpesti;u} is not a uniform-in-$m\ L^m$-estimate for $u$; 
 taking the limit as $m\to \infty$ 
in the $L^m$-estimate for $u$ obtained in Lemma \ref{lem;Lpesti;u} 
 does not directly enable us to 
have an $L^\infty$-estimate for $u$. 
Thus we employ the Moser iteration argument 
to have an $L^\infty$-estimate 
by using the $L^m$-estimate for $u$ for $m\ge 1$. 
%
%
\begin{lem}\label{lemma6.5}
There exists a constant $C>0$ such that
\begin{align*}
\sup_{t\in(0,\tmax)}\|u(\cdot,t)\|_{L^{\infty}(\Omega)}\leq C,
\end{align*}
i.e., $\limsup_{t\nearrow\tmax}\|u(\cdot,t)\|_{L^{\infty}(\Omega)}
< \infty$.
\end{lem}
\begin{proof}
We put 
\begin{align}\label{mk}
m_k:=2^k+(p-1)
\end{align}
 for $k\in \mathbb{N}\cup\{0\}$.
Then we can verify that 
\[
m_0=p\geq 1, \quad m_k>m_{k-1}\quad \mbox{for all} \ 
k\in \mathbb{N}\cup\{0\}\] 
\mbox{and} 
\[m_k\to \infty \quad \mbox{as}\ k\to \infty\]
as well as 
\[ m_{k-1}=\frac{m_k-p+1}{2}.\]
Now, given $T\in (0,\tmax)$, we introduce 
\begin{align*}
M_k:=\sup_{t\in(0,T)}\int_{\Omega}u^{m_k}(x,t)\,dx
\end{align*} 
for an arbitrary integer $k$ and let $m:=m_k$. 
 First we can use the
  Gagliardo--Nirenberg type inequality (see \cite{L-L}) and find $c>0$ 
 such that
 \begin{align}\label{GNineq-ump1}
 \|u^{m+p-1}\|_{L^1(\Omega)} 
 \leq c \|\nabla u^{m+p-1}\|^a_{L^1(\Omega)} 
 \|u^{m+p-1}\|^{1-a}_{L^\frac{1}{2}(\Omega)} 
 +c\|u^{m+p-1}\|_{L^\frac{1}{2}(\Omega)}, 
 \end{align}
where $a:=\frac{n}{n+1}$ for all $t\in(0,\tmax)$. 
Moreover, the Young inequality enables us to see that 
 \begin{align}\label{Yineq-ump1}
 c & \left(\int_{\Omega}|\nabla u^{m+p-1}|\right)^a
\|u^{m+p-1}\|^{1-a}_{L^\frac{1}{2}(\Omega)} 
  \\ 
       \nonumber
&\leq \frac{a}{2}\int_{\Omega}u^{m+p-2}|\nabla u|
+(1-a)\{2(m+p-1)^\frac{a}{1-a}c^\frac{1}{1-a}
\|u^{m+p-1}\|_{L^\frac{1}{2}(\Omega)} 
\\ 
       \nonumber
&\leq \frac{1}{2}\int_{\Omega}u^{m+p-2}|\nabla u|
+\{2c(m+p-1)\}^{n+1}
\|u^{m+p-1}\|_{L^\frac{1}{2}(\Omega)}. 
 \end{align}
Thus plugging \eqref{GNineq-ump1} and \eqref{Yineq-ump1} 
into \eqref{lem6.3} implies 
\begin{align*}
 \frac{d}{dt}&\int_{\Omega}u^m+\int_{\Omega}u^{m}
 \\   
  \nonumber
  &\leq 
  Cm^2\left(\frac{m+p-1}{2}\right)^{n+1}
\|u^{m+p-1}\|_{L^\frac{1}{2}(\Omega)} 
+C^m_1+C_2 m+C_3m\cdot  C^m_4, 
\end{align*}
with $C:={4c}^{n+1}$. 
Therefore we apply a comparison argument to establish that 
\begin{align*}
M_k \leq 
 \max\left\{\int_{\Omega} u^{m_k}_0, 
 2( C^{m_k}_1+C_2m_k+C_3m_k\cdot C^{m_k}_4),
 2C m^{n+3}_kM^2_{k-1}\right\}
\end{align*}   
for all $t\in(0,\tmax)$. Now if there exists a sequence 
$(m_{k_j})_{j\in \mathbb{N}}$ such that $m_{k_j}\to \infty$ and 
\begin{align*}
M_{k_j} \leq \int_{\Omega} u^{m_{k_j}}_0
\end{align*}
for all $j\in \mathbb{N}$, then we take the $m_{k_j}$-th 
root of the both sides to obtain 
\begin{align*}
\sup_{t\in(0,T)}\|u(\cdot,t)\|_{L^{m_{k_j}}(\Omega)}\leq \|u_0\|_{L^{m_{k_j}}(\Omega)}.
\end{align*}
We derive from letting $m_{k_j}\to \infty$ that
\begin{align*}
\sup_{t\in(0,T)}\|u(\cdot,t)\|
_{L^{\infty}(\Omega)}\leq \|u_0\|_{L^{\infty}(\Omega)}
\end{align*}
in this case. Contrarily, if there is no such sequence, at the first  we have
\begin{align*}
M_k
\leq 2( C^{m_k}_1+C_2m_k+C_3m_k\cdot C^{m_k}_4).
\end{align*}
We take the $m_k$-th root on the
 both sides and use the 
elementary inequality $\sqrt[m]{X+Y}\leq\sqrt[m]{X}+\sqrt[m]{Y}$, where $X\geq 0$ and $Y\geq 0$, to obtain 
\begin{align*}
\sup_{t\in(0,T)}\|u(\cdot,t)\|_{L^{m_k}(\Omega)}
&\leq 
\left\{2( C^{m_k}_1+C_2m_k+C_3m_k\cdot C^{m_k}_4)\right\}
^\frac{1}{m_k}
  \\
   \nonumber
&\leq
2^\frac{1}{m_k}C_1+ C_2^\frac{1}{m_k} m_k^\frac{1}{m_k}
+C_4C_3^\frac{1}{m_k} m_k^\frac{1}{m_k}.
\end{align*}
By taking $m_k\to \infty$, we obtain 
\begin{align*}
\sup_{t\in(0,T)}\|u(\cdot,t)\|_{L^{\infty}(\Omega)}\leq C_1+1+C_4.
\end{align*} 
In the last case we will use the Moser iteration argument. 
The definition of $m_k$ (see \eqref{mk}) and the elementary 
inequality $2^k+(p-1)\leq p2^k$ yields that 
\begin{align*}
M_k &\leq 
 2C\{2^k+(p-1)\}^{n+3}M^2_{k-1}
 \\   \nonumber
&\leq 
 2Cp^{n+3}(2^{n+3})^k M^2_{k-1}.
\end{align*}
Then there exists a constant $b>1$ independent of $T$ which satisfies
  \begin{align*}
  M_k\leq b^kM_{k-1} \quad \mbox{for all}\  k\geq 1.
  \end{align*}
  Using the same argument as in the proof of
  \cite[Lemma 6.2]{B-W}, we have
  \begin{align*}
  M_k\leq b^{2^{k+1}}M^{2^k}_0 \quad \mbox{for all}\  k\geq 1.
  \end{align*}
  Thus, we take the $m_k$-th root of the
  both sides and use 
  \eqref{mk} again to obtain  
  \begin{align*}
  \sup_{t\in(0,T)}\|u(\cdot,t)\|_{L^{m_k}(\Omega)}
= M_k^\frac{1}{m_k}
\leq b^\frac{2^{k+1}}{2^k
+(p-1)} M_0^\frac{2^k}{2^k+(p-1)}.
  \end{align*}
  Therefore, taking $k\to \infty$, we establish
  \begin{align*}
\sup_{t\in(0,T)}\|u(\cdot,t)\|_{L^{\infty}(\Omega)}\leq 
 b^2M_0
\end{align*} 
and arrive at the conclusion.
\end{proof}
%
%
\begin{proof}[{\bf Proof of Theorem 1.2.}]
Thanks to Lemma \ref{lemma6.5} and 
 extensibility criterion obtained in Theorem 1.1, 
 we see that $\tmax=\infty$ and that there exists $C>0$ 
such that 
\begin{align*}
\|u(\cdot,t)\|_{L^\infty(\Omega)}\leq C\quad \mbox{for all}\ t>0,
\end{align*} 
which means the end of the proof.
\end{proof}

\end{document}